\documentclass[11pt]{amsart}

\usepackage{amsmath, amssymb, amsfonts, amsthm, xcolor,mathabx,stmaryrd}
\usepackage{scalerel}
\usepackage{comment}
\definecolor{my-link}{rgb}{0.5,0.0,0.0}
\definecolor{my-blue}{rgb}{0.0,0.0,0.6}
\definecolor{my-red}{rgb}{0.5,0.0,0.0}
\definecolor{my-green}{rgb}{0.2,0.5,0.2}
\definecolor{darkgreen}{rgb}{0.0,0.5,0.0}
\definecolor{darkblue}{rgb}{0.0,0.0,0.3}
\definecolor{light-gray}{gray}{0.7}
\RequirePackage[colorlinks,urlcolor=my-blue,linkcolor=my-red,citecolor=my-green]{hyperref}
\usepackage[numbers,comma,square,sort&compress]{natbib}

\usepackage{enumitem}

\usepackage[normalem]{ulem}

\usepackage{datetime}

\newcommand{\geo}{\Gamma}

\newcommand{\from}[1]{_{{#1}}}
\newcommand{\dir}[3]{^{{#1, #2#3}}}
\newcommand{\baddir}{\Theta^\w}
\newcommand{\tht}{\theta}
\newcommand{\sig}{{\raisebox{1pt}{\scaleobj{0.6}{\boxempty}}}}  
\newcommand{\sigg}{{\raisebox{-0.2pt}{\scaleobj{1}{\boxempty}}}}  
\newcommand{\siggg}{{\raisebox{0.5pt}{\scaleobj{0.7}{\boxempty}}}}  

\newcommand{\h}{\mathfrak h}
\newcommand{\jumps}{\mathcal V}
\newcommand{\Ijumps}{\mathcal J}

\def\acoef{\bar a}
\def\dirsp{\mathfrak V}

\newcommand{\R}{\mathbb{R}}
\newcommand{\Z}{\mathbb{Z}}
\newcommand{\N}{\mathbb{N}}

\renewcommand{\P}{\mathbb{P}}

\newcommand{\e}{\varepsilon}

\newcommand{\LL}{\mathcal{L}}

\newcommand{\w}{\omega}

\newcommand\abullet{\hspace{0.6pt}{\raisebox{1pt}{\scaleobj{0.6}{\bullet}}}\hspace{0.8pt}}  

\newtheorem{thm}{Theorem}[section]
\newtheorem{lem}[thm]{Lemma}
\newtheorem{prop}[thm]{Proposition}

\theoremstyle{definition}
\newtheorem{defn}[thm]{Definition}

\theoremstyle{remark}
\newtheorem{rmk}[thm]{Remark}

\numberwithin{figure}{section}
\numberwithin{equation}{section}

\usepackage[margin=1in]{geometry}

\title[The Martin boundary of the Directed Landscape]{The Martin boundary of the Directed Landscape}

\author[F.~Rassoul-Agha]{Firas Rassoul-Agha}
\address{Firas Rassoul-Agha\\ University of Utah\\  Mathematics Department\\ 155S 1400E\\   Salt Lake City, UT 84112\\ USA.}
\email{firas@math.utah.edu}
\urladdr{http://www.math.utah.edu/~firas}
\thanks{F.\ Rassoul-Agha was partially supported by National Science Foundation grants DMS-2054630 and DMS-2450951}

\author[M.~Sweeney]{Mikhail Sweeney}
\address{Mikhail Sweeney\\ University of Utah\\  Mathematics Department\\ 155S 1400E\\   Salt Lake City, UT 84112\\ USA.}
\email{sweeney@math.utah.edu}
\thanks{M.\ Sweeney was partially supported by F.\ Rassoul-Agha through National Science Foundation grant DMS-2054630}




\date{Last updated: October 15, 2026}

\begin{document}

\begin{abstract}
In the directed landscape, the Martin boundary coincides with the horofunction boundary. We show that functions in this boundary are precisely the eternal solutions possessing a spatial growth rate, and that the minimal Martin boundary is given by the Busemann functions. Moreover, every eternal solution can be expressed as a max-plus convex combination of countably many Busemann functions. Horofunctions are exactly those eternal solutions that admit a representation in terms of at most two Busemann functions with a common growth rate. As a consequence of instability, not all horofunctions are Busemann functions, and the Martin boundary is strictly larger than its minimal part.
\end{abstract}

\maketitle

\section{Introduction}

We study the Kardar--Parisi--Zhang (KPZ) fixed point, the central object of the \emph{KPZ universality class}. This broad class is believed to encompass a wide range of models, including one-dimensional stochastic Hamilton--Jacobi equations (both viscous and inviscid), interacting particle systems, percolation and growth processes, random polymer measures, driven diffusive systems, and random matrix models.  
Models in this class are typically interpreted as randomly growing height functions $h:\mathbb{R}^2\to\mathbb{R}$ which, after suitable centering and rescaling,
\begin{align}\label{KPZscaled}
\varepsilon^{1/2}\bigl(h(\varepsilon^{-1}x,\varepsilon^{-3/2}t)-C\,\varepsilon^{-3/2}t\bigr),
\end{align}
converge as $\varepsilon\to 0$ to a universal limit $\mathfrak{h}$, known as the \emph{KPZ fixed point}. This convergence has been rigorously established for a number of models 
\cite{Wu-26,Agg-Cor-Heg-24-a-,Agg-Cor-Heg-24-b-,Nic-Qua-Rem-20,Vir-20-,Dau-Vir-21-,Zha-25-,Vet-Vir-26}.  

The KPZ fixed point was first constructed in \cite{Mat-Qua-Rem-21} via an explicit description of its transition probabilities in terms of Fredholm determinants. Subsequently, \cite{Dau-Ort-Vir-22} provided a variational characterization through the construction of the \emph{directed landscape}, a continuous real-valued random field $\{\mathcal{L}(x,s;y,t): x,s,y,t\in\mathbb{R},\, t>s\}$, defined on a Polish space $(\Omega,\mathfrak S,\P)$, 
and satisfying 
\begin{align}\label{composition}
\mathcal{L}(x,s;y,t)=\sup_{z\in\mathbb{R}}\bigl(\mathcal{L}(x,s;z,r)+\mathcal{L}(z,r;y,t)\bigr), \quad t>r>s.
\end{align}
This field serves as the random environment in which the KPZ fixed point evolves: given a terminal condition $\varphi$ at time $t$, the fixed point is represented by the variational formula
\begin{align}\label{KPZfp}
\mathfrak{h}(x,s)=\sup_{y\in\mathbb{R}}\bigl\{\mathcal{L}(x,s;y,t)+\varphi(y)\bigr\}, \qquad s<t.
\end{align}
(Using terminal rather than initial conditions is a convention; it aligns with viewing the directed landscape as a continuum last-passage percolation model with forward semi-infinite geodesics.)

A canonical example within the KPZ universality class is the \emph{KPZ equation}, introduced in \cite{Kar-Par-Zha-86} as a model for the evolution of a one-dimensional random interface:
\begin{align}\label{KPZ}
\partial_t h = \tfrac{\lambda}{2}(\partial_x h)^2 + \tfrac{\nu}{2}\partial_{xx}h + \sigma W,
\end{align}
where $W$ denotes space-time white noise, $\sigma\ne 0$ controls the noise strength, $\nu>0$ is a viscosity parameter, and $\lambda\ne0$ modulates the strength of the nonlinearity. By the scaling properties of the KPZ equation (see, e.g., Remark 1.1 in \cite{Jan-Ras-Sep-23-1F1S-}), the rescaled field
$(x,t)\mapsto \nu\lambda^{-1}h\bigl(\nu^{-3}\lambda^2\sigma^2 x,\ \nu^{-5}\lambda^4\sigma^4 t\bigr)$
solves \eqref{KPZ} with parameters $\lambda=\nu=\sigma=1$. 
With this normalization, \cite{Vir-20-} showed that the centered and rescaled field \eqref{KPZscaled} converges as $\e\searrow 0$ to the KPZ fixed point. Applying the same scaling once more, one finds that the field $\varepsilon^{1/2}h(\varepsilon^{-1}x,\varepsilon^{-3/2}t)$ solves \eqref{KPZ} with $\lambda=1$, $\nu=\varepsilon^{1/2}$, and $\sigma=\varepsilon^{1/4}$.  
Since the KPZ equation is a prototypical stochastic Hamilton--Jacobi equation with rough forcing, the KPZ fixed point may thus be viewed as a prototype of an inviscid (albeit degenerate) stochastic Hamilton--Jacobi equation. From this perspective, \eqref{KPZfp} can be interpreted as a Hopf--Lax--Oleinik variational formula. See the introduction of \cite{Ras-Swe-26-b-} for more. 

In light of the conjectured universality of models in the KPZ class, one expects the conclusions of this paper to remain valid beyond the KPZ fixed point, applying more broadly to one-dimensional stochastic Hamilton--Jacobi equations, both inviscid and viscous, as well as to planar first-passage percolation (which corresponds to a random metric on $\Z^2$).

We next give a broad description of the main questions addressed in this work and our principal results. 
From \eqref{KPZfp}, one sees that $\h$ satisfies the dynamic programming (or update) equation
\begin{align}\label{h-update}
\h(x,s)=\sup\bigl\{\mathcal{L}(x,s;y,t)+\h(y,t):y\in\R\bigr\}, \qquad s<t.
\end{align}
Letting the terminal time $t\to\infty$ leads naturally to the notion of an \emph{eternal solution}, that is, a function $\h:\R^2\to\R$ satisfying \eqref{h-update} for all pairs $s<t$. 
Viewing $\mathcal L$ as an operator, such solutions may be interpreted as \emph{harmonic functions} in the max-plus algebra.

At the same time, \eqref{KPZfp} provides a geometric interpretation: viewing it as a Hopf--Lax--Oleinik variational formula (see \eqref{LPP} below), the maximizing paths correspond to characteristic curves, which in the directed landscape picture are geodesics of a continuum last-passage percolation model. In this correspondence, semi-infinite characteristic curves of an eternal solution become semi-infinite geodesics. Borrowing ideas from metric geometry, one can associate to such geodesics \emph{Busemann functions}, obtained as limits of differences of passage times. These functions are themselves eternal solutions, and hence provide canonical examples of max-plus harmonic functions.

A parallel perspective arises from potential theory. In positive temperature models (such as random walk in random environment), differences of passage times correspond to ratios of hitting probabilities or Green’s functions, whose compactification leads to the \emph{Martin boundary}. In the present zero-temperature setting, the analogous construction yields a Martin boundary in the max-plus algebra. From the geometric viewpoint, this same procedure produces \emph{horofunctions}, which generalize Busemann functions and again yield eternal solutions.

Harmonic functions form a convex set in the max-plus sense, which leads to the natural question of whether they can be represented as max-plus convex combinations of a minimal family of harmonic functions. In the language of potential theory, this family is the minimal Martin boundary.

Our results connect these perspectives in the setting of the KPZ fixed point and the directed landscape. First, we show that every eternal solution can be expressed as a max-plus countable convex combination of Busemann functions. Second, we prove that Busemann functions are extremal harmonic functions and coincide with the extremal horofunctions, thereby identifying them with the minimal Martin boundary. Third, we characterize horofunctions (equivalently, points of the Martin boundary) as those eternal solutions that have a spatial growth rate, and show that they are convex combinations of at most two Busemann functions sharing that rate. Finally, we provide a geometric description of these solutions. In the case of two Busemann components, the space-time plane is partitioned by a bi-infinite interface separating the regions where the solution agrees with each component. More generally, an arbitrary eternal solution admits a decomposition into countably many such regions, separated by interfaces that coalesce backward in time.

We note that related results appear in \cite{Bha-Bus-Sor-25-} and in the concurrent preprint \cite{Bha-Bus-Sor-26-}, where the competition interfaces introduced in \cite{Rah-Vir-25} are used to construct the interfaces described in the previous paragraph and to establish the convex decomposition of eternal solutions. By contrast, our approach treats eternal solutions directly within a broader framework relating Busemann functions, horofunctions, and the Martin boundary. In our work, the competition interfaces (Propositions \ref{I-prop} and \ref{prop:general-I}) play only a limited technical role, serving primarily as bookkeeping devices in the proofs of Theorems \ref{thm:convcomb} and \ref{thm:general-I}, respectively.\smallskip

The questions we consider here have their origins in metric geometry, in particular in the study of compactifications of noncompact spaces. On general metric spaces, horofunctions arise from the metric compactification \cite{Gro-81}, 
while Busemann functions form a distinguished subclass constructed from infinite geodesics \cite{Bus-55}. 
A central question in this area is whether all horofunctions are Busemann functions. 
The answer depends on the space: it is affirmative, e.g., for 
Hadamard spaces \cite[Proposition 2.5]{Bal-95},
but fails in other settings, such as Teichm\"uller spaces with complex dimension at least two 
\cite[Theorem 1.1]{Miy-14}.

In the presence of additional structure, other compactifications exist. For instance, on harmonic Riemannian manifolds or graphs equipped with a random walk, the Martin boundary provides a harmonic compactification, with its minimal part corresponding to the extremal points of the convex set of normalized positive harmonic functions \cite{Mar-41}. When both structures are available, it is natural to ask how these compactifications relate: do they coincide, and if so, are Busemann functions precisely the extremal harmonic functions? Again, the answer varies by setting. In particular, for Gromov hyperbolic manifolds \cite[Theorem 6.2]{Anc-90} 
and for noncompact, nonflat, simply connected harmonic manifolds \cite[Theorem 1 and Proposition 2]{Zim-14}, 
the two notions coincide, and moreover, in both cases, the entire Martin boundary is minimal.

In our setting, the passage time $\mathcal L(x,s;y,t)$ is defined only for $t>s$ and therefore does not define a metric. Nevertheless, one can introduce natural analogues of horofunctions, Busemann functions, and the Martin boundary, and pose the same questions. The closest precedents to our work are \cite{Con-01} and \cite{Aki-Gau-Wal-09}, where analogous connections were developed in the settings of weak KAM theory in Hamiltonian dynamics and, respectively, max-plus stochastic operators (the zero-temperature limit of Markov chains). A key difference, however, is that all of the aforementioned settings are deterministic, whereas in our case the landscape $\mathcal L$ is random, corresponding to a random geometry or to Hamilton--Jacobi equations with random forcing. A principal consequence of this randomness is the emergence of instability, manifested in the existence of multiple Busemann functions associated with the same growth rate. Our results show that this instability implies that not all horofunctions are Busemann functions; rather, the latter correspond precisely to the extremal harmonic functions, and the Martin boundary is strictly larger than its minimal part. For comparison, in the hypothetical absence of instability, every horofunction would be a Busemann function and the Martin boundary would be minimal.
\smallskip

{\bf Notation.}
$\Z$ is the set of integers, $\N$ is the set of positive integers, and $\R$ is the set of real numbers.
Given real numbers $x,s,y,t$, we use $s\wedge t=\min(s,t)$, $s\vee t=\max(s,t)$, and write $(x,s) \le (y,t)$ to mean $x \le y$ and $s \le t$. 
Then $(x,s) < (y,s)$ means $x < y$. 
Given an interval $I\subset\R$, a \emph{space-time path} is a  function $\gamma:I\to\R^2$ such that for all $t\in I$, the second coordinate of $\gamma(t)$ is equal to $t$. 
It will at times be convenient to abuse notation and regard $\gamma$
as the path $\gamma(I)$. Given a subinterval $J\subset I$, $\gamma|_J$ is the restriction of $\gamma$ to $J$.
Given two intervals $I_i\subset\R$ and two space-time paths $\gamma_i:I_i\to\R^2$, $i\in\{1,2\}$, $\gamma_1\preceq\gamma_2$
means $\gamma_1(t)\le\gamma_2(t)$
for all $t\in I_1\cap I_2$.

\section{Setting and main results}\label{sec:main}

Taking the limit $t \to \infty$ in \eqref{KPZfp}, we arrive at the following definition. Given a realization of the directed landscape $\mathcal L$, an \emph{eternal solution} is a function $\h:\R^2\to\R$ such that
\begin{align}\label{h-eternal}
\h(x,s)=\sup_{y\in\R}\{\mathcal L(x,s;y,t)+\h(y,t)\},
\qquad \text{for all } s<t \text{ and $x$ in $\R$}.
\end{align}

\begin{rmk}
One could allow $\h$ to take values in $[-\infty,\infty]$. However, if $\h(x,s)=-\infty$ for some $(x,s)\in\R^2$, then, by \eqref{h-eternal}, $\h(y,t)=-\infty$ for all $t>s$ and $y\in\R$. Fixing such a $t$ and applying \eqref{h-eternal} again, we get $\h(z,r)=-\infty$ for all $r<t$ and $z\in\R$, and hence $\h\equiv -\infty$.
Likewise, if $\h(y,t)=\infty$ for some $(y,t)\in\R^2$, then $\h(x,s)=\infty$ for all $s<t$ and $x\in\R$. Thus, apart from these degenerate cases, $\h$ is finite everywhere. For this reason, we restrict attention to real-valued eternal solutions.
\end{rmk}

Lemma \ref{eternal-cvx} below shows that the collection of eternal solutions forms a cone in the max-plus algebra (more formally, a semimodule over the semiring $(\R\cup\{-\infty\},\max,+)$). In particular, if $\h$ is an eternal solution, then so is $\h+c$ for any $c\in\R$. Consequently, upon fixing a normalization---for instance, by requiring $\h(x_0,s_0)=0$ for some $(x_0,s_0)\in\R^2$---the set of normalized eternal solutions becomes a convex set in the max-plus sense. A natural question is whether a Choquet-type theorem holds, namely, whether every normalized eternal solution can be expressed as a max-plus convex combination of extremal ones. Theorem \ref{thm:main1} below answers this question in the affirmative. 

Part of answering the above question is to identify the extremal eternal solutions. To this end, we consider several perspectives on the directed landscape, beginning by viewing it as a last-passage percolation model.

Given a space-time path $\gamma:[s,t]\to\R^2$, define its passage time by
\begin{align}\label{L(gamma)}
\mathcal L(\gamma)
=\inf_{k\in\N}\;\inf_{s=r_0<r_1<\cdots<r_{k-1}<r_k=t}
\sum_{i=1}^k \mathcal L\bigl(\gamma(r_{i-1});\gamma(r_i)\bigr).
\end{align}
    By  \cite[Corollary 10.7]{Dau-Ort-Vir-22}, almost surely, there exists a finite (random) constant $C>0$ such that 
    \begin{align}\label{L-bound}
    \Bigl|\mathcal L(x,s;y,t)+\frac{(y-x)^2}{t-s}\Bigr|\le
    C|t-s|^{1/3}\log^{4/3}\Bigl(\frac{2\lVert(x,s,y,t)\rVert+2}{t-s}\Bigl)\log^{2/3}(\lVert(x,s,y,t)\rVert+2),
    \end{align}
    for all $x,s,y,t\in\R^2$ with $t>s$. Here, $\lVert(x,s,y,t)\rVert=\sqrt{x^2+s^2+y^2+t^2}$.
Because of this bound, if $\gamma$ is not continuous, then $\mathcal L(\gamma)=-\infty$. 
Thus,
\begin{align}\label{LPP}
\mathcal L(x,s\,;y,t)= \sup_{\gamma}\mathcal L(\gamma),
\end{align}
where the supremum can be restricted to all continuous space-time paths $\gamma:[s,t]\to\R^2$ with $\gamma(s)=(x,s)$ and $\gamma(t)=(y,t)$. This formulation interprets the directed landscape as a continuum space-time last-passage percolation model. Accordingly, a path attaining the supremum in \eqref{LPP} is called a \emph{point-to-point geodesic} from $(x,s)$ to $(y,t)$. These continuous paths exist by Lemma 13.2 in \cite{Dau-Ort-Vir-22}.
As $t\to\infty$, point-to-point geodesics converge to semi-infinite paths, which we call \emph{semi-infinite geodesics}. More precisely, a semi-infinite geodesic from $(x,s)\in\R^2$ is a continuous space-time path $\gamma:[s,\infty)\to\R^2$ with $\gamma(s)=(x,s)$ such that $\gamma$ is a geodesic between any two of its points.

The next set of properties of semi-infinite geodesics can be anticipated from the connection to Hamilton--Jacobi equations. Indeed, via \eqref{LPP},  \eqref{KPZfp} can be viewed as a Hopf--Lax--Oleinik variational formula, with $\mathcal L(\gamma)$ being the action functional. In this interpretation, maximizing geodesics correspond to characteristic lines, and semi-infinite geodesics correspond to characteristic lines of eternal solutions. This perspective suggests considering eternal solutions that possess a conserved spatial growth rate (or asymptotic velocity), in the sense that $\lim_{|x|\to\infty} x^{-1}\h(x,t)$ exists and is independent of $t$. In this case, the associated characteristic lines are expected to be directed. Moreover, one expects these characteristic lines to exhibit hyperbolicity, which in the directed landscape manifests as coalescence. We now state the rigorous properties of semi-infinite geodesics. These properties all hold on a single full probability event.

Theorem 2.5(i) of \cite{Bus-Sep-Sor-24} asserts that, almost surely, every semi-infinite geodesic $\gamma$ is $\tht$-directed for some $\tht\in\R$, that is, $t^{-1}\gamma(t)\to(\tht,1)$ as $t\to\infty$.
Combining results from \cite{Bus-Sep-Sor-24,Bus-25-,Dau-25}, one can organize the family of all semi-infinite geodesics into a process
\begin{align}\label{geodesics}
\bigl\{\geo\from{(x,s)}\dir{S}{\tht}{\sig}:(x,s)\in\R^2,\,S\in\{L,M,R\},\,\tht\in\R,\,\sigg\in\{-,+\}\bigr\}.
\end{align}
Almost surely, for any $x,s,\tht\in\R$, the geodesics 
$\geo\from{(x,s)}\dir{L}{\tht}{-}$ and $\geo\from{(x,s)}\dir{R}{\tht}{+}$ 
are, respectively, the leftmost and rightmost $\tht$-directed geodesics from $(x,s)$. By \cite[Lemma B.1]{Ras-Swe-26-b-}, planarity induces an ordering on these paths: almost surely,
for all $s$, $x<y<z$, and $\tht$ in $\R$, 
and all $\sigg\in\{-,+\}$ and $S\in\{L,M,R\}$,
\begin{align}
\geo\from{(x,s)}\dir{R}{\tht}{\sig}\preceq\geo\from{(y,s)}\dir{L}{\tht}{\sig}
\preceq\geo\from{(y,s)}\dir{M}{\tht}{\sig}
\preceq\geo\from{(y,s)}\dir{R}{\tht}{\sig}
\preceq\geo\from{(z,s)}\dir{L}{\tht}{\sig}
\quad\text{and}\quad
\geo\from{(x,s)}\dir{S}{\tht}{-}\preceq\geo\from{(x,s)}\dir{S}{\tht}{+}.\label{geo:mono}
\end{align}Depending on the configuration, the three geodesics $\geo\from{(x,s)}\dir{S}{\tht}{\sig}$, $S\in\{L,M,R\}$, may coincide, only two may be distinct, or all three may be distinct. 

The distinction between leftmost and rightmost semi-infinite geodesics corresponds to the presence of shock points, familiar from inviscid Hamilton--Jacobi equations: points $(x,s)$ from which $\geo\from{(x,s)}\dir{L}{\tht}{\sig}$ and $\geo\from{(x,s)}\dir{R}{\tht}{\sig}$ immediately separate. The paper \cite{Ras-Swe-26-b-} provides a detailed analysis of these points and the configurations of semi-infinite geodesics emanating from them.

Crucially, the sign distinction is not merely a technicality but an essential feature of the model. Namely, by (3.12) and (3.29) in \cite{Ras-Swe-26-b-}, for almost every $\w$ there exists an $\w$-dependent countable dense set $\baddir \subset \R$ such that, for all $x,s \in \R$ and $S \in \{L,M,R\}$,
\begin{align}\label{sign}
\begin{split}
&\tht\notin\baddir\Longrightarrow\geo\from{(x,s)}\dir{S}{\tht}{-}=\geo\from{(x,s)}\dir{S}{\tht}{+}\text{ (in which case we may drop the sign and write $\geo\from{(x,s)}\dir{S}{\tht}{}$)},\\
&\tht\in\baddir\Longrightarrow\geo\from{(x,s)}\dir{S}{\tht}{-}(t)<\geo\from{(x,s)}\dir{S}{\tht}{+}(t)\quad\text{for all large enough }t.
\end{split}
\end{align}

The geodesics in \eqref{geodesics} form coalescing families: By Theorem 7.1 and Remark 7.2 in \cite{Bus-Sep-Sor-24}, for all $x, s, y, t, \tht \in \R$,  $\sigg \in \{-,+\}$, and $S, S' \in \{L, M, R\}$, we have that
\begin{align}
\begin{split}
&\text{if }\geo\from{(x,s)}\dir{S}{\tht}{\sig}(r)
=
\geo\from{(y,t)}\dir{S'}{\tht}{\sig}(r)
\text{ for }r \ge s \vee t \text{ with }r> s\wedge t\text{, then }
\geo\from{(x,s)}\dir{S}{\tht}{\sig}\big|_{[r,\infty)}
=\geo\from{(y,t)}\dir{S'}{\tht}{\sig}\big|_{[r,\infty)}.
\end{split}\label{geo:coal1}\\[4pt]
&\text{Furthermore, there exists an $r$ as in \eqref{geo:coal1}.}
\label{geo:coal2}
\end{align}

This coalescence allows one to recover the eternal solution whose characteristic lines are the geodesics $\geo\from{(x,s)}\dir{S}{\tht}{\sig}$, $(x,s)\in\R^2$, $S\in\{L,M,R\}$. Specifically, for $x_0,s_0,x,s,\tht\in\R$ and $\sigg\in\{-,+\}$, define
\begin{align}\label{W-def}
W^{\tht\sig}(x,s;x_0,s_0)
=
\lim_{r\to\infty}\Bigl(\mathcal L\bigl(x,s;\geo\from{(x_0,s_0)}\dir{S}{\tht}{\sig}(r)\bigr)
-
\mathcal L\bigl(x_0,s_0;\geo\from{(x_0,s_0)}\dir{S}{\tht}{\sig}(r)\bigr)\Bigr).
\end{align}
The limit exists due to \eqref{geo:coal2}; in fact, it agrees with the prelimit expression for any $r$ past the coalescence time. Moreover, the definition is independent of the choice of $S$. 
Using \eqref{composition}, one checks \cite[Theorem 5.1(iv)]{Bus-Sep-Sor-24} that, almost surely, for any $x_0,s_0,x,s,\tht\in\R$, $\sigg\in\{-,+\}$, and $t>s$,
\begin{align}\label{W-eternal}
W^{\tht\sig}(x,s;x_0,s_0)=\sup_{y\in\R}\{\mathcal L(x,s;y,t)+W^{\tht\sig}(y,t;x_0,s_0)\}.
\end{align}
Thus, $(x,s)\mapsto W^{\tht\sig}(x,s;x_0,s_0)$ is an eternal solution that vanishes at $(x_0,s_0)$. By Lemma 5.12(iv) in \cite{Bus-Sep-Sor-24}, this solution has spatial growth rate $2\tht$: almost surely, for all $x_0,s_0,s,\tht\in\R$ and $\sigg\in\{-,+\}$,
\begin{align}\label{Wgrowth}
\lim_{|x|\to\infty} x^{-1} W^{\tht\sig}(x,s;x_0,s_0) = 2\tht.
\end{align}

We again see the significance of the sign distinction. By \cite[(5.7)]{Bus-Sep-Sor-24}, almost surely, for any $x_0,s_0$,
\begin{align}\label{exceptional}
\tht\in\baddir 
\;\Longleftrightarrow\;
\exists\, x\in\R:\, W^{\tht-}(x,s_0;x_0,s_0)\ne W^{\tht+}(x,s_0;x_0,s_0).
\end{align}
Thus, when $\tht\in\baddir$, there exist at least two distinct eternal solutions with the same spatial growth rate $2\tht$. The region where the two solutions differ is called the \emph{instability region}. See \cite{Ras-Swe-26-b-} for a detailed description of this region, its properties, and its relation to the shock points.

A natural question concerns the role of these special eternal solutions. To shed light on this, we turn to a geometric interpretation. By \eqref{composition}, the passage time $\mathcal L(x,s;y,t)$ is superadditive:
\begin{align}\label{superadditivity}
\mathcal L(x,r;z,t)
\ge
\mathcal L(x,r;y,s)+\mathcal L(y,s;z,t),
\qquad \text{for all } x,y,z \in \R \text{ and } s<r<t.
\end{align}
If not for the restriction $s<r<t$, the quantity $-\mathcal L(x,s;y,t)$ would define a random signed metric on $\R^2$. In this sense, the special eternal solutions introduced above are analogues of the functions studied by Busemann in \cite{Bus-55}, and we refer to them as \emph{Busemann eternal solutions}. Accordingly, we call the function $W^{\tht\sig}:\R^4\to\R$ a \emph{Busemann function}.
Viewed as functions of four variables, the functions $W^{\tht\sig}$ satisfy the cocycle property
\begin{align}\label{cocycle}
W^{\tht\sig}(x,s;z,r)+W^{\tht\sig}(z,r;y,t)=W^{\tht\sig}(x,s;y,t),
\quad \forall x,s,z,r,y,t,\tht\in\R,\ \sigg\in\{-,+\},
\end{align}
and are therefore also referred to as \emph{Busemann cocycles}. 
We now state our first main result.

\begin{thm}\label{thm:main1}
The following holds on a full probability event.
All eternal solutions can be represented as a max-plus convex mixture of countably many Busemann eternal solutions. Furthermore, all Busemann eternal solutions are extremal.
\end{thm}

The proof is given in Section \ref{sec:proofs}. The decomposition into Busemann eternal solutions is in Theorem \ref{thm:decomp}, while extremality comes from Theorem \ref{thm:min}. The fact that only countably many Busemann eternal solutions are required follows from parts \eqref{prop:general-I.a} and \eqref{thm:general-I.d} of Theorem \ref{thm:general-I} (see Remark \ref{rk:countable}).

In addition, Theorem \ref{thm:general-I} provides a detailed description of eternal solutions in terms of the Busemann eternal solutions appearing in their convex decomposition. This description is given via a tree of interfaces that coalesce backward in time and partition space-time into regions where the solution coincides with the corresponding Busemann eternal solutions.
We also address the maximality and the uniqueness of the decomposition, in Theorem \ref{thm:h-rep-maximal}.\smallskip

Having defined Busemann functions, we next follow \cite{Gro-81} and define \emph{horofunctions} (also known as \emph{generalized Busemann functions}) as limit points of the family of functions
\[
(x,s)\longmapsto \mathcal L(x,s;z_r,r)-\mathcal L(x_0,s_0;z_r,r),
\]
as $r\to\infty$, with convergence understood in the sense of local uniform convergence. 
The distinction between the two notions is that, in the definition of Busemann functions, the endpoints $(z_r,r)$ lie along a semi-infinite geodesic, whereas for horofunctions they are arbitrary. In the metric geometry literature, the set of horofunctions is known as the \emph{horofunction boundary}. Theorem \ref{Buslim} verifies that horofunctions are also eternal solutions. 
As explained in the introduction, a central question in metric geometry is whether all horofunctions are Busemann functions. 
Our second main result shows that this is not the case and provides a characterization of the set of horofunctions.

\begin{thm}\label{thm:main2}
The following holds on an event of full probability. A function $\h:\R^2\to\R$ is a horofunction if and only if it is an eternal solution and there exists $\tht\in\R$ such that $\h$ has spatial growth rate $2\tht$, that is,
for all $t\in\R$, $x^{-1}\h(x,t)\to2\tht$ as $|x|\to\infty$.
Moreover, if $\h$ is normalized to vanish at $(x_0,s_0)\in\R^2$, then it has spatial growth rate $2\tht$ if and only if there exist $a_-,a_+\in\R\cup\{-\infty\}$ with $a_-\vee a_+=0$ such that
\[\h(x,s)=(W^{\tht-}(x,s;x_0,s_0)+a_-)\vee(W^{\tht+}(x,s;x_0,s_0)+a_+), \quad \forall (x,s)\in\R^2.\]
\end{thm}

The proof of Theorem \ref{thm:main2} is given in Section \ref{sec:proofs}. The second equivalence follows from Theorem \ref{thm:eternal+growth->horo} together with Lemma \ref{lm:h-growth-rate}, while the first equivalence follows from Theorems \ref{Buslim} and \ref{thm:convcomb}.\smallskip

We conclude this section with an interpretation of our results from the perspective of potential theory, which also explains the title of the paper.

Equation \eqref{h-eternal} shows that $\h$ is a fixed point of $\mathcal L$ in the max-plus sense, that is, it solves the max-plus linear equation $\mathcal L \h = \h$. Accordingly, eternal solutions are precisely the \emph{$\mathcal L$-harmonic} functions in the max-plus algebra. For simplicity, we suppress the explicit reference to the max-plus algebra and to the operator $\mathcal L$, and refer to such functions simply as \emph{harmonic}.

Viewing $\mathcal L(x,s;y,t)$ as the analogue of a Green's kernel, the collection of horofunctions plays the role of the Martin boundary familiar from the theory of Markov processes. 
The subset consisting of extremal harmonic functions is thus called the \emph{minimal Martin boundary}. Theorem \ref{thm:main1} shows that, in the directed landscape model, the minimal Martin boundary is given by the Busemann functions and that all harmonic functions are their max-plus convex combinations. Theorem \ref{thm:main2} further shows that, due to the presence of instability (i.e., $\baddir \ne \varnothing$), the full Martin boundary---consisting of all horofunctions---is strictly larger.


\section{Proofs}\label{sec:proofs}

Before proceeding, we record a useful monotonicity property of the Busemann functions, which follows from the geodesics ordering \eqref{geo:mono}: by Theorem 5.1(iii) in \cite{Bus-Sep-Sor-24}, on a full probability event, for any $s$, $x<y$, and $\tht<\tht'$ in $\R$, 
\begin{align}\label{W:mono}
W^{\tht'+}(x,s;y,s)\le W^{\tht'-}(x,s;y,s)\le W^{\tht+}(x,s;y,s)\le W^{\tht-}(x,s;y,s).
\end{align}

We now proceed with the development of the proofs of Theorems \ref{thm:main1} and \ref{thm:main2}. First, observe that \emph{max-plus linear combinations} of eternal solutions are again eternal solutions. Thus, the collection of eternal solutions is a max-plus cone (formally, a semimodule over the semiring $(\R\cup\{-\infty\},\max,+)$).

\begin{lem}\label{eternal-cvx}
Suppose $\h_i$, $i\in I$, is a family of eternal solutions. Let $a_i\in[-\infty,\infty)$, $i\in I$. Then 
\begin{align}\label{h-comb}
\h(x,s)=\sup_{i\in I}(\h_i(x,s)+a_i)
\end{align}
satisfies \eqref{h-eternal}.
\end{lem}

\begin{proof} 
Write 
\begin{align*}
\h(x,s) 
&= \sup_{i\in I}(\h_i(x,s)+a_i)
=\sup_{i\in I}\sup_{y\in\R} \{\mathcal{L}(x,s;y,t)+\h_i(y,t)+a_i\}\\
&=\sup_{y\in\R} \{\mathcal{L}(x,s;y,t)+\h(y,t)\}.
\qedhere
\end{align*}
\end{proof}

A \emph{max-plus convex mixture} is an expression of the form \eqref{h-comb} with $\sup_i a_i = 0$. If the index set $I$ is countable, such a mixture is also called a \emph{convex combination}.

\begin{rmk}\label{rk:h-finite?}
For $\h$ in \eqref{h-comb} to be an eternal solution, we need to have $\h(x,s)\in\R$ for each $x,s\in\R$. 
This cannot be guaranteed for a general choice of the coefficient $a_i$.
If, e.g., $I$ is finite and $\h_i$ are all finite, then so is $\h$. See Remark \ref{rk:h-finite?2} for more on this.
\end{rmk}

Adding a constant to an eternal solution gives an eternal solution. Therefore, 
we can normalize an eternal solution $\h$ by requiring $\h(x_0,s_0)=0$ for some reference point $(x_0,s_0)$. 
This makes the set of (normalized) eternal solutions convex (in the max-plus algebra).  
An eternal solution $\h$ with $\h(x_0,s_0)=0$ is said to be \emph{extremal} if having $\h=(f+a)\vee(g+b)$ with $a\vee b=0$ and $f,g$ two eternal solutions with $f(x_0,s_0)=g(x_0,s_0)=0$ implies $f=\h$ or $g=\h$. 
We aim to identify the extreme eternal solutions and show that all (normalized) eternal solutions can be written as convex mixtures of the extreme ones. To achieve  this, we use the associated semi-infinite geodesics, which we now define. 

\begin{defn}\label{def:h-geo}
Given an eternal solution $\h$, a space-time path  $\gamma:[s,\infty)\to\R^2$ is an $\h$-geodesic if 
\begin{align}\label{L=h}
\mathcal L(\gamma(r);\gamma(t))=\h(\gamma(r))-\h(\gamma(t))\quad\text{for all }t>r\ge s.
\end{align}
\end{defn}

Theorem 5.9(iii) in \cite{Bus-Sep-Sor-24} asserts that for $S\in\{L,R\}$, $\tht\in\R$, $\sigg\in\{-,+\}$, and $(x,s)\in\R^2$, $\geo\from{(x,s)}\dir{S}{\tht}{\sig}$ is a $W^{\tht\sig}$-geodesic. Then the coalescence \eqref{geo:coal2} implies that the same holds for $S=M$.

We next justify calling $\gamma$ a geodesic, in Definition \ref{def:h-geo}.

\begin{lem}\label{lm:h-geo is geo}
   The following holds on a full probability event.
   Suppose $\h$ is an eternal solution. Let $s\in\R$. 
Let $\gamma:[s,\infty)\to\R^2$ be an $\h$-geodesic. Then, for any $t>r\ge s$, $\gamma|_{[r,t]}$ is a geodesic from $\gamma(r)$ to $\gamma(t)$.  Consequently, $\gamma$ is continuous.
\end{lem}

\begin{proof}
Take any partition $r=s_0<s_1<\cdots<s_n=t$. Then, by \eqref{L=h},
    \[\sum_{i=0}^{n-1}\mathcal L(\gamma(s_i);\gamma(s_{i+1}))
     =\sum_{i=0}^{n-1}(\h(\gamma(s_i))-\h(\gamma(s_{i+1}))
    =\h(\gamma(r))-\h(\gamma(t))\ge\mathcal L(\gamma(r);\gamma(t)),\]
where the last inequality is from \eqref{h-eternal}. 
Together with the superadditivity \eqref{superadditivity}, we get that $\mathcal L$ is additive along $\gamma|_{[s,t]}$. Thus, this path is a geodesic between its endpoints. 
\end{proof}

The next lemma justifies calling $\gamma$ an $\h$-geodesic. It follows directly from \eqref{h-eternal} and \eqref{L=h}.

\begin{lem}\label{lm:maximizer}
   The following holds on a full probability event.
   Suppose $\h$ is an eternal solution. Let $s\in\R$. 
Let $\gamma:[s,\infty)\to\R^2$ be an $\h$-geodesic. Then, for any $t>r\ge s$, setting $u=y$ such that $(y,t)=\gamma(t)$ maximizes $u\mapsto\mathcal L(\gamma(r);u,t)+\h(u,t)$. 
\end{lem}

To ensure the existence of $\h$-geodesics, we need to invoke a sub-quadratic growth bound. 

\begin{thm}\label{th:h-growth}
On a full-probability event, any eternal solution $\h$ satisfies
\begin{align}\label{h-growth}
    \lim_{|y|\to\infty} |y|^{-1}\bigl(\sup_{t'\le t\le t''}\h(y,t)-ay^2\bigr) = -\infty
    \qquad \text{for all $t'<t''$ and $a>0$ in $\R$}.
\end{align}
Furthermore, for any $x$ and $s<t$ in $\R$,
       \begin{align}\label{L+h bound}
       \lim_{|y|\to\infty}|y|^{-1}\bigl(\mathcal L(x,s;y,t)+\h(y,t)\bigr)=-\infty.
       \end{align}
Consequently, there exists a finite $K=K(\h,\w,x,s,t)>0$ such that for any $x$ and $s<t$ in $\R$,
\[\sup_{y\in\R}\{\mathcal L(x,s;y,t) + \h(y,t)\}=\sup_{|y|\le K}\{\mathcal L(x,s;y,t) + \h(y,t)\}.\]
\end{thm}

\begin{proof}
    Suppose that \eqref{h-growth} fails for some $a>0$ and $t_1<t_2$. Then there exist $b>0$ and sequences $t_n\in[t',t'']$ and $y_n$ with $|y_n|\to\infty$ such that $\h(y_n,t_n)-a y_n^2 \ge -b|y_n|$ for all integers $n$.
    Set $s = t' - 1/a$. Then $t_n-s\ge1/a$ for all $n$. Choose $x$ with $|x| > b/(2a)$. By passing to a subsequence if necessary, we may assume that $y_n$ have the same sign for all $n$. Take $x$ to have that sign. Then, 
    \[-b|y_n| + 2a x y_n = -b|y_n| + 2a|x|\,|y_n| = (2a|x|-b)\,|y_n| \to \infty.\]
    Use \eqref{h-eternal} and \eqref{L-bound} to write
    \begin{align*}
    \h(x,s)&\ge \h(y_n,t_n)+\mathcal L(x,s;y_n,t_n)\\
    &\ge ay_n^2-b|y_n|-a(y_n-x)^2-o(|y_n|)=
    -b|y_n|+2axy_n-ax^2-o(|y_n|).
    \end{align*}
    Taking $n\to\infty$ gives $\h(x,s)=\infty$, contradicting the requirement of finiteness of eternal solutions. This proves \eqref{h-growth}.

  Applying \eqref{L-bound} and \eqref{h-growth} with $t'=t''=t$ and $a=1/(t-s)$, we have that for any $b>0$, there exists $c>0$ such that for all $|y|\ge c$,
    \begin{align*}
      \mathcal L(x,s;y,t)+\h(y,t)
       \le -a(y-x)^2+o(|y|)+a y^2-b|y|
       =2a yx-a x^2-b|y|+o(|y|).
     \end{align*}
    Divide by $|y|$, take it to $\infty$, then take $b\to\infty$ to get \eqref{L+h bound}.
\end{proof}

With the growth bound \eqref{h-growth}, Theorem \ref{thm:h-continuous} implies the spatial continuity of eternal solutions.

\begin{thm}\label{th:h-cont}
On a full probability event, any eternal solution is continuous on $\R^2$.
\end{thm}

We can now define leftmost and rightmost $\h$-geodesics and describe their properties.

\begin{defn}\label{def:LR h-geo}
Given an eternal solution $\h$, a starting point $(x,s)\in\R^2$, and $S\in\{L,R\}$, set $\geo\from{(x,s)}^{S,h}(s)=(x,s)$. For $t>s$, let $\geo\from{(x,s)}^{L,h}(t)$ and $\geo\from{(x,s)}^{R,h}(t)$ be, respectively, the infimum and the supremum of the set of maximizers of $y\mapsto\mathcal L(x,s;y,t)+\h(y,t)$.
\end{defn}

\begin{thm}\label{thm:h-geo}
   On a full probability event, the following statements hold for any eternal solution $\h$ and all $(x,s)\in\R^2$.
   \begin{enumerate} [label={\rm(\alph*)}, ref={\rm\alph*}]   \itemsep=3pt 
   \item\label{thm:h-geo.a} For any $t\ge s$, $-\infty<\geo\from{(x,s)}^{L,h}(t)\le\geo\from{(x,s)}^{R,h}(t)<\infty$.
   \item\label{thm:h-geo.b}  For any $S\in\{L,R\}$ and any $t>s$, setting $u=y$ such that $(y,t)=\geo\from{(x,s)}^{S,h}(t)$ maximizes $u\mapsto\mathcal L(x,s;u,t)+\h(u,t)$.
   \item\label{thm:h-geo.c}  For any $S\in\{L,R\}$, $\geo\from{(x,s)}^{S,h}$ is an $\h$-geodesic and is therefore a continuous space-time path.
   \item\label{thm:h-geo.d} For any $S\in\{L,R\}$ and $t\ge r\ge s$, $\geo\from{\geo\from{(x,s)}^{S,h}(r)}^{S,h}(t)=\geo\from{(x,s)}^{S,h}(t)$.
   \item\label{thm:h-geo.e} For any $z,r\in\R$, $S\in\{L,R\}$, and $t\ge s\vee r$, if $\geo\from{(x,s)}^{S,h}(t)=\geo\from{(z,r)}^{S,h}(t)$, then $\geo\from{(x,s)}^{S,h}(t')=\geo\from{(z,r)}^{S,h}(t')$ for all $t'\ge t$.
   \item\label{thm:h-geo.f} For any $z>x$ and $t\ge s$ in $\R$, for any $S\in\{L,R\}$, $\geo\from{(x,s)}^{S,h}(t)\le\geo\from{(z,s)}^{S,h}(t)$.
   \item\label{thm:h-geo.g} For any $t>t'\ge s$, $\geo\from{(x,s)}^{L,h}\big|_{[t',t]}$ is the leftmost geodesic between its endpoints and $\geo\from{(x,s)}^{R,h}\big|_{[t',t]}$ is the rightmost geodesic between its endpoints.
   \item\label{thm:h-geo.h} For any $\h$-geodesic $\gamma$ out of $(x,s)$ and any $t\ge s$, $\geo\from{(x,s)}\dir{L}{h}{}(t)\le\gamma(t)\le\geo\from{(x,s)}\dir{R}{h}{}(t)$.
   \item\label{thm:h-geo.i} Suppose $\gamma_n:[s,\infty)\to\R^2$ is a sequence of $\h$-geodesics that converges pointwise to $\gamma$. Then $\gamma$ is a semi-infinite $\h$-geodesic. 
   \item\label{thm:h-geo.j} For each $\e>0$, there exists $\delta>0$ such that if $z,z'\in(x-\delta,x]$, then  $\geo\from{(z,s)}\dir{L}{h}{}(t)=\geo\from{(x,s)}\dir{L}{h}{}(t)$ and $\geo\from{(z,s)}\dir{R}{h}{}(t)=\geo\from{(z',s)}\dir{R}{h}{}(t)$ for all $t\ge s+\e$ and if $z\in[x,x+\delta)$, then  $\geo\from{(z,s)}\dir{R}{h}{}(t)=\geo\from{(x,s)}\dir{R}{h}{}(t)$ and $\geo\from{(z,s)}\dir{L}{h}{}(t)=\geo\from{(z',s)}\dir{L}{h}{}(t)$ for all $t\ge s+\e$.
   \end{enumerate}
\end{thm}

\begin{proof}
    Theorems \ref{th:h-growth} and \ref{th:h-cont} imply that the maximizers exist and are bounded. This gives part \eqref{thm:h-geo.a}. 
    Part \eqref{thm:h-geo.b} comes from the continuity of $\mathcal L(x,s;\abullet,t)$ and $\h(t,\abullet)$.\smallskip
    
    Part \eqref{thm:h-geo.c}. We prove the claim for $S=R$, the proof being similar for $S=L$. 
   Take $t>s$ and let $\gamma$ be any geodesic from $(x,s)$ to $\geo\from{(x,s)}^{R,h}(t)$. Take $r\in(s,t)$. By part \eqref{thm:h-geo.b}, we we have 
    \begin{align}\label{aux586}
    \mathcal L\bigl(x,s;\geo\from{(x,s)}^{R,h}(t)\bigr)+h\bigl(\geo\from{(x,s)}^{R,h}(t)\bigr)=\h(x,s).
    \end{align}
    Since $\gamma$ is a geodesic, we also have
    \begin{align}\label{aux588}
    \mathcal L(x,s;\gamma(r))+\mathcal L(\gamma(r);\geo\from{(x,s)}^{R,h}(t))=\mathcal L\bigl(x,s;\geo\from{(x,s)}^{R,h}(t)\bigr).
    \end{align}
    Putting the two equations together, we get
    \begin{align}\label{aux409}
    \mathcal L(x,s;\gamma(r))+\mathcal L(\gamma(r);\geo\from{(x,s)}^{R,h}(t))+h\bigl(\geo\from{(x,s)}^{R,h}(t)\bigr)=\h(x,s).
    \end{align}

 On the other hand, \eqref{h-eternal} implies
   \[\mathcal L(x,s;\gamma(r))+\h(\gamma(r))\le \h(x,s)
   \quad\text{and}\quad
   \mathcal L\bigl(\gamma(r);\geo\from{(x,s)}^{R,h}(t)\bigr)+h\bigl(\geo\from{(x,s)}^{R,h}(t)\bigr)\le \h(\gamma(r)).\]
Together with \eqref{aux409}, we get that both of the above inequalities are in fact equalities. Thus 
\begin{align}\label{aux416}
\text{$u$ such that }(u,r)=\gamma(r)\text{ maximizes }u\mapsto\mathcal L(x,s;u,r)+\h(u,r).
\end{align}   
    
Now abbreviate
$(z,r)=\geo\from{(x,s)}^{R,h}(r)$ and $(y,t)=\geo\from{(z,r)}^{R,h}(t)$ and let $\tau$ be any geodesic from $(z,r)$ to $(y,t)$. From \eqref{aux416} and the definition of $\geo\from{(x,s)}^{R,h}(r)$, we have 
\begin{align}\label{aux453}
\gamma(r)\le(z,r)=\tau(r).
\end{align}
By part \eqref{thm:h-geo.b},
\begin{align}\label{aux:611}
\mathcal L(x,s;z,r)+\h(z,r)=\h(x,s)\quad\text{and}\quad
\mathcal L(z,r;y,t)+\h(y,t)=\h(z,r).
\end{align}
From this, we get 
$\mathcal L(x,s;z,r)+\mathcal L(z,r;y,t)+\h(y,t)=\h(x,s)$.
But \eqref{h-eternal} implies that $\h(x,s)\ge\mathcal L(x,s;y,t)+\h(y,t)$. Thus, we have 
$\mathcal L(x,s;z,r)+\mathcal L(z,r;y,t)\ge\mathcal L(x,s;y,t)$.
Together with the reverse inequality \eqref{superadditivity}, we get equality. Now we have
$\mathcal L\big(x,s;y,t)+\h(y,t)=\h(x,s)$,
which says that $u=y$ maximizes $u\mapsto\mathcal L(x,s;u,t)+h(u,t)$. Consequently, 
\begin{align}\label{aux467}
\tau(t)=(y,t)\le\geo\from{(x,s)}^{R,h}(t)=\gamma(t).
\end{align}
This, \eqref{aux453}, and the continuity of the paths imply that they must intersect at some time $r'\in[r,t]$. 

We continue the argument under the assumption that $r'\in(r,t]$, the argument when $r'=r$ being an easier version. 
By the same argument that led to \eqref{aux416}, we have 
$\mathcal L(z,r;\tau(r'))+\h(\tau(r'))=\h(z,r)$. 
This and the first equality in \eqref{aux:611} give
\[\mathcal L(x,s;z,r)+\mathcal L(z,r;\tau(r'))+\h(\tau(r'))=\h(x,s).\]
But \eqref{h-eternal} implies $\h(x,s)\ge\mathcal L(x,s;\tau(r'))+\h(\tau(r'))$. Thus, $\mathcal L(x,s;z,r)+\mathcal L(z,r;\tau(r'))\ge\mathcal L(x,s;\tau(r'))$.
This and superadditivity \eqref{superadditivity} give 
\begin{align}\label{aux433}
\mathcal L(x,s;z,r)+\mathcal L(z,r;\tau(r'))=\mathcal L(x,s;\tau(r')).
\end{align}

Since $\gamma$ is a geodesic, 
$\mathcal L(x,s;\gamma(r'))+\mathcal L(\gamma(r');\gamma(t))=\mathcal L(x,s;\gamma(t))$.
Together with \eqref{aux433} and $\gamma(r')=\tau(r')$, we get
$\mathcal L(x,s;z,r)+\mathcal L(z,r;\gamma(r'))+\mathcal L(\gamma(r');\gamma(t))=\mathcal L(x,s;\gamma(t))$.
Then, by \eqref{superadditivity},
\[\mathcal L(x,s;z,r)+\mathcal L(z,r;\gamma(t))\ge\mathcal L(x,s;\gamma(t)).\]
Using \eqref{superadditivity} again, we get that this must be an equality. 

Recalling that $(z,r)=\geo\from{(x,s)}^{R,h}(r)$ and $\gamma(t)=\geo\from{(x,s)}^{R,h}(t)$ and applying \eqref{aux586} and \eqref{aux:611}, we can write
\begin{align}
\mathcal L\bigl(\geo\from{(x,s)}^{R,h}(r);\geo\from{(x,s)}^{R,h}(t)\bigr)
&=\mathcal L(z,r;\gamma(t))=\mathcal L(x,s;\gamma(t))-\mathcal L(x,s;z,r)\notag\\
&=[\mathcal L(x,s;\gamma(t))+\h(\gamma(t))]-[\mathcal L(x,s;z,r)+\h(z,r)]+\h(z,r)-\h(\gamma(t))\notag\\
&=\h(x,s)-\h(x,s)+h\bigl(\geo\from{(x,s)}^{R,h}(r)\bigr)-h\bigl(\geo\from{(x,s)}^{R,h}(t)\bigr).\label{aux657}
\end{align}
This proves part \eqref{thm:h-geo.c}.\smallskip

Part \eqref{thm:h-geo.d}. Consider $t>r>s$. 
From \eqref{aux467}, $\geo\from{\geo\from{(x,s)}^{R,h}(r)}^{R,h}(t)=(y,t)\le\geo\from{(x,s)}^{R,h}(t)$. On the other hand, \eqref{aux657} implies that $u$ such that $(u,t)=\geo\from{(x,s)}^{R,h}(t)$ maximizes $u\mapsto\mathcal L\bigl(\geo\from{(x,s)}^{R,h}(r);u,t\bigr)+h(u,t)$. This implies $\geo\from{(x,s)}^{R,h}(t)\le\geo\from{\geo\from{(x,s)}^{R,h}(r)}^{R,h}(t)$.
Consequently, we have equality and part \eqref{thm:h-geo.d} is proved for the case $S=R$. The case $S=L$ works the same way.\smallskip

Part \eqref{thm:h-geo.e} follows from part \eqref{thm:h-geo.d}. Part \eqref{thm:h-geo.f} follows from part \eqref{thm:h-geo.e} and the continuity of the paths.\smallskip

Part \eqref{thm:h-geo.g}. We prove the claim for $S=R$, $S=L$ being similar. When $t'=s$, the claim follows from \eqref{aux416} and the fact that $\geo\from{(x,s)}^{R,h}(r)$ is the supremum of all the maximizers of $u\mapsto\mathcal L(x,s;u,r)+\h(u,r)$. The case $t'>s$ follows from the case $t'=s$ and the restarting formula from part \eqref{thm:h-geo.d}. 
\smallskip

Part \eqref{thm:h-geo.h} follows directly from the definitions: the path $\gamma(t)$ satisfies 
$\mathcal L(x,s;\gamma(t),t)+\mathfrak{h}(\gamma(t),t)=\mathfrak{h}(x,s)$,
and therefore maximizes $y\mapsto \mathcal L(x,s;y,t)+\h(y,t)$. By definition, $\geo\from{(x,s)}^{L,h}(t)$ and $\geo\from{(x,s)}^{R,h}(t)$ are, respectively, the infimum and supremum of the set of such maximizers.\smallskip

Part \eqref{thm:h-geo.i} follows from Definition \ref{def:h-geo}, the continuity of $\mathcal L$, and the continuity of $\h$ in Lemma \ref{th:h-cont}.\smallskip

Part \eqref{thm:h-geo.j}. We prove the claim for $z \nearrow x$, the other case being analogous. By part \eqref{thm:h-geo.f}, for any $S \in \{L,R\}$ and $t \ge s$, $\geo\from{(z,s)}^{S,h}(t)$ is nondecreasing in $z$. Hence, these geodesics converge pointwise to a limit. By part \eqref{thm:h-geo.i}, this limit is an $\h$-geodesic $\gamma$.
Applying part \eqref{thm:h-geo.f} again, we obtain in the case $S = L$ that $\gamma \preceq \geo\from{(x,s)}^{L,h}$. Together with part \eqref{thm:h-geo.h}, this implies that $\gamma = \geo\from{(x,s)}^{L,h}$. When $S = R$, the limit $\gamma$ is still an $\h$-geodesic, but need not coincide with $\geo\from{(x,s)}^{L,h}$ or $\geo\from{(x,s)}^{R,h}$.
In either case, the limit is continuous, and thus Dini's theorem \cite[Theorem 7.13]{Rud-76} implies convergence is locally uniform. The claim now follows from Theorem 1.18 in \cite{Bat-Gan-Ham-22} and part \eqref{thm:h-geo.e}.
\end{proof}

Since \eqref{geodesics} covers all the semi-infinite geodesics, we have that, on a full probability event, for any $\h$-geodesic out of $(y,t)\in\R^2$, there exist $\tht\in\R$ and $\sigg\in\{-,+\}$ such that this is a $W^{\tht\sig}$-geodesic. By \eqref{geo:coal2}, this geodesic will coalesce with $\geo\from{(x_0,s_0)}\dir{S'}{\tht}{\sig}$, for any $(x_0,s_0)\in\R^2$ and $S'\in\{L,M,R\}$. From this, we next show that any eternal solution is a convex mixture of Busemann eternal solutions. 

Consider $x_0,s_0\in\R$ and an eternal solution $\h$ with $\h(x_0,s_0)=0$. Take $\tht\in\R$ and $\sigg\in\{-,+\}$.
If there exist $(y,t)\in\R^2$ and an $\h$-geodesic  out of $(y,t)$ that coalesces with some (and hence any) $W^{\tht\sig}$-geodesic out of $(x_0,s_0)$, then let
\begin{align}\label{def:a}
a^{\tht\sig}_{(x_0,s_0)}(\h)=\h(y,t)-W^{\tht\sig}(y,t;x_0,s_0).
\end{align}
If such a geodesic does not exist, then set $a^{\tht\sig}_{(x_0,s_0)}(\h)=-\infty$.  Definition \eqref{def:a} is independent of $(y,t)$, as the next lemma shows.

\begin{lem}
The following holds on a full probability event.
For any $x_0,s_0\in\R$, any eternal solution $\h$ with $\h(x_0,s_0)=0$, and any $\tht\in\R$ and $\sigg\in\{-,+\}$, $a^{\tht\sig}_{(x_0,s_0)}(\h)$ does not depend on the choice of $(y,t)$. Furthermore, $a^{\tht\sig}_{(x_0,s_0)}(\h)\in[-\infty,0]$.
\end{lem}

\begin{proof}
It suffices to consider the case where $a^{\tht\sig}_{(x_0,s_0)}(\h)>-\infty$. In this case, there exists a $(y,t)$ and an $\h$-geodesic $\gamma$ out of $(y,t)$ that coalesces with a $W^{\tht\sig}$-geodesic out of $(x_0,s_0)$. By the coalescence \eqref{geo:coal2}, all $W^{\tht\sig}$-geodesics out of $(x_0,s_0)$ coalesce with $\geo\from{(x_0,s_0)}\dir{L}{\tht}{\sig}$. Therefore, it is enough to consider this particular geodesic out of $(x_0,s_0)$.

 Suppose $(x,s)\in\R^2$ is another point from which an $\h$-geodesic coalesces with $\geo\from{(x_0,s_0)}\dir{L}{\tht}{\sig}$. By the coalescence \eqref{geo:coal2}, $\geo\from{(x,s)}\dir{L}{\tht}{\sig}$ and $\geo\from{(y,t)}\dir{L}{\tht}{\sig}$ also coalesce with $\geo\from{(x_0,s_0)}\dir{L}{\tht}{\sig}$. Then there exists a point $(z,r)$, $r>s\vee t$, that belongs simultaneously to these three geodesics and to the two $\h$-geodesics. Then
 \begin{align*}
    \h(x,s)-\h(y,t)
    &=\bigl(\mathcal L(x,s;z,r)+\h(z,r)\bigr)-\bigl(\mathcal L(y,t;z,r)+\h(z,r)\bigr)\\
    &=\bigl(\mathcal L(x,s;z,r)-\mathcal L(x_0,s_0;z,r)\bigr)-\bigl(\mathcal L(y,t;z,r)-\mathcal L(x_0,s_0;z,r)\bigr)\\
    &=W^{\tht\sig}(x,s;x_0,s_0)-W^{\tht\sig}(y,t;x_0,s_0).
    \end{align*}
Rearranging gives 
\begin{align}\label{aux707}
\h(x,s)-W^{\tht\sig}(x,s;x_0,s_0)=\h(y,t)-W^{\tht\sig}(y,t;x_0,s_0),
\end{align}
which proves the first claim. 
To see the second claim, take $(x,s)$, with $s>s_0$, to be a point on both $\geo\from{(x_0,s_0)}\dir{L}{\tht}{\sig}$ and $\gamma$. Such a point exists because the two geodesics coalesce.
The restriction $\gamma|_{[s,\infty)}$ is an $\h$-geodesic out of $(x,s)$ that coalesces with $\geo\from{(x_0,s_0)}\dir{L}{\tht}{\sig}$. Thus, we can apply \eqref{aux707} to get
\begin{align*}
a^{\tht\sig}_{(x_0,s_0)}(\h)
&=\h(y,t)-W^{\tht\sig}(y,t;x_0,s_0)
=\h(x,s)-W^{\tht\sig}(x,s;x_0,s_0)\\
&=\h(x,s)+W^{\tht\sig}(x_0,s_0;x,s)
=\h(x,s)+\mathcal L(x_0,s_0;x,s)\le \h(x_0,s_0)=0.\qedhere
\end{align*}
\end{proof}

We now have all the ingredients for the convex decomposition of eternal solutions.

\begin{thm}\label{thm:decomp}
    The following holds on a full probability event.
    For any $x_0,s_0\in\R$, and any eternal solution $\h$ with $\h(x_0,s_0)=0$, we have for all $x,s\in\R$
    \begin{align}\label{h-rep}
    \h(x,s)=\sup_{\tht\in\R,\siggg\in\{-,+\}}\{W^{\tht\sig}(x,s;x_0,s_0)+a^{\tht\sig}_{(x_0,s_0)}(\h)\}
    \end{align}
    and
    \begin{align}\label{supa}
    \sup_{\tht\in\R,\siggg\in\{-,+\}}a^{\tht\sig}_{(x_0,s_0)}(\h)=0.
    \end{align}
\end{thm}

\begin{proof}
First consider $\tht\in\R$ and $\sigg\in\{-,+\}$ such that $a^{\tht\sig}_{(x_0,s_0)}(\h)>-\infty$. Then there exists $(y,t)\in\R^2$ and an $\h$-geodesic out of $(y,t)$ that coalesces with $\geo\from{(x_0,s_0)}\dir{L}{\tht}{\sig}$. By \eqref{geo:coal2}, $\geo\from{(x_0,s_0)}\dir{L}{\tht}{\sig}$ coalesces with $\geo\from{(x,s)}\dir{L}{\tht}{\sig}$. Let $(z,r)$ be the coalescence point of all three geodesics.
Using it in the definition \eqref{def:a} we get
\begin{align*}
W^{\tht\sig}(x,s;x_0,s_0)+a^{\tht\sig}_{(x_0,s_0)}(\h)
&=W^{\tht\sig}(x,s;x_0,s_0)+\h(z,r)-W^{\tht\sig}(z,r;x_0,s_0)\\
&=W^{\tht\sig}(x,s;z,r)+\h(z,r)
=\mathcal L(x,s;z,r)+\h(z,r)\le \h(x,s).
\end{align*}
Taking a supremum over $\tht\in\R$ and $\sigg\in\{-,+\}$ proves one part of \eqref{h-rep}:
\begin{align}\label{aux-ge}
\h(x,s)\ge\sup_{\tht\in\R,\siggg\in\{-,+\}}\{W^{\tht\sig}(x,s;x_0,s_0)+a^{\tht\sig}_{(x_0,s_0)}(\h)\}.
\end{align}

For the other direction, consider an $\h$-geodesic $\pi$ out of $(x,s)$. Since \eqref{geodesics} includes all the semi-infinite geodesics, there exist $\tht\in\R$, $\sigg\in\{-,+\}$, and $S\in\{L,M,R\}$ such that $\pi=\geo\from{(x,s)}\dir{S}{\tht}{\sig}$. By the coalescence \eqref{geo:coal2}, $\pi$ coalesces with $\geo\from{(x_0,s_0)}\dir{S}{\tht}{\sig}$. Therefore,
$a^{\tht\sig}_{(x_0,s_0)}(\h)=\h(x,s)-W^{\tht\sig}(x,s;x_0,s_0)$.
Rearranging, we get
\[\h(x,s)=W^{\tht\sig}(x,s;x_0,s_0)+a^{\tht\sig}_{(x_0,s_0)}(\h).\]
This and \eqref{aux-ge} give \eqref{h-rep}. Evaluating \eqref{h-rep} at $(x_0,s_0)$ and using $\h(x_0,s_0)=0$ yields \eqref{supa}.
\end{proof}

An eternal solution $\h$ is said to be \emph{minimal} if for any eternal solution $g \le \h$, there exists $c \in \R$ such that $\h = g + c$. Geometrically, this means that the portion of the (max-plus) cone of eternal solutions lying below $\h$ consists precisely of the (max-plus) line segment joining $-\infty$ and $\h$. 
In the classical $(+,\times)$ setting, minimality and extremality are equivalent. The next lemma shows that, in max-plus algebra, minimality is equivalent to a stronger notion of extremality.

\begin{lem}\label{eternal:min=ext}
    The following holds on a full probability event.
    For all $x_0,s_0\in\R$, 
    an eternal solution $\h$ with $\h(x_0,s_0)=0$ is minimal if and only if every representation $\h=(f+a)\vee(g+b)$, where $f$ and $g$ are eternal solutions with $f(x_0,s_0)=g(x_0,s_0)=0$, and $a,b\in\R$ satisfy $a\vee b=0$, is trivial in the sense that $f=g=\h$. Consequently, every minimal eternal solution is extremal.
\end{lem} 

\begin{proof}
Suppose first that $\h$ is a minimal and that 
$\h=(g+a)\vee(f+b)$ with $a,b\in\R$, $a\vee b=0$ and the eternal solutions $g$ and $f$ satisfying $g(x_0,s_0)=f(x_0,s_0)=0$.
Then $g+a\le \h$ and $f+b\le \h$ and we must have $\h=g+a+c=f+b+c'$ for some $c,c'\in\R$. 
$f(x_0,s_0)=g(x_0,s_0)=\h(x_0,s_0)=0$ implies  $c=-a$ and $c'=-b$ and consequently $f=g=\h$.

Conversely, suppose $\h$ satisfies the decomposition property. Suppose $g$ is an eternal solution such that $g\le \h$. ($g$ need not satisfy $g(x_0,s_0)=0$.) 
Set $c=-g(x_0,s_0)\ge-\h(x_0,s_0)=0$. Since $g$ is finite-valued, $c\in\R$.
Then $g+c$ is an eternal solution that vanishes at $(x_0,s_0)$.
Moreover, $\h=\h\vee((g+c)-c)$.
Applying the decomposition property with $f=\h$, $a=0$, and $b=-c\le0$, we obtain $g+c=\h$. 
\end{proof}

Before proving the next theorem, we need the following preliminary lemma.

\begin{lem}\label{aux:349}
The following holds on a full probability event.
Let $x_0,s_0,\tht,\tht'\in\R$ and $\sigg,\sigg'\in\{-,+\}$. Suppose that for some $a\in\R$, $W^{\tht\sig}(x,s;x_0,s_0)-W^{\tht'\sig'}(x,s;x_0,s_0)\le a$ for all $x,s\in\R$. Then $\tht'=\tht$ and, if $\tht\in\baddir$, then also $\sigg'=\sigg$. 
\end{lem}

\begin{proof}
From the assumption, we have that 
\begin{align}\label{aux:356}
W^{\tht\sig}(x,s_0;x_0,s_0)-W^{\tht'\sig'}(x,s_0;x_0,s_0)\le a\quad\text{for all }x\in\R.
\end{align}
Take $x<x_0$ and divide through by $x$. Take $x\to-\infty$ and apply \eqref{Wgrowth} to get that $\tht-\tht'\ge 0$. A symmetric argument, taking $x>x_0$ then $x\to\infty$, gives $\tht-\tht'\le0$. This proves that $\tht=\tht'$.  

By \cite[(5.6)]{Bus-Sep-Sor-24}, on a full probability event, for any $\tht\in\baddir$,
\begin{align}\label{(5.6)}
&\lim_{x\to\sig\infty}\bigl(W^{\tht+}(x,s_0;x_0,s_0)-W^{\tht-}(x,s_0;x_0,s_0)\bigr)=\sigg\infty,\quad\sigg\in\{-,+\}.
\end{align}
This and \eqref{aux:356} imply that we must have $\sigg=\sigg'$.
\end{proof}

The next theorem says that Busemann eternal solutions are extremal.

\begin{thm}\label{thm:min}
The following holds on a full probability event.
For any $x_0,s_0,\tht\in\R$, and $\sigg\in\{-,+\}$, the eternal solution $\h(x,s)=W^{\tht\sig}(x,s;x_0,s_0)$ is minimal and hence extremal.   
\end{thm}

\begin{proof}
 Suppose that $g$ is an eternal solution such that $g\le \h$.
Let $a=-g(x_0,s_0)$. Then $f=g+a$ is an eternal solution that is normalized to satisfy $f(x_0,s_0)=0$. 
By \eqref{h-rep}, we have for all $(x,s)\in\R^2$, $\tht'\in\R$, and $\sigg'\in\{-,+\}$,
\[W^{\tht\sig}(x,s;x_0,s_0)=\h(x,s)
\ge g(x,s)
= f(x,s)-a
\ge W^{\tht'\sig'}(x,s;x_0,s_0)+a^{\tht'\sig'}_{(x_0,s_0)}(f)-a.\]
This, together with Lemma \ref{aux:349}, shows us that if $a^{\tht'\sig'}_{(x_0,s_0)}(f)>-\infty$, then necessarily $\tht'=\tht$, and, if $\tht\in\baddir$, also $\sigg'=\sigg$. It follows that 
\begin{align}\label{aux819}
W^{\tht'\sig'}=W^{\tht\sig}\quad\text{and}\quad 
a^{\tht'\sig'}_{(x_0,s_0)}(f)=a^{\tht\sig}_{(x_0,s_0)}(f).
\end{align}
(The same conclusion holds when $\tht\notin\baddir$, since in that case there is no sign distinction.) Now, \eqref{supa} implies $a^{\tht\sig}_{(x_0,s_0)}(f)=0$. This, \eqref{h-rep}, and \eqref{aux819} gives us that, for all $x,s\in\R$,
\begin{align*}
    g(x,s)+a
    &=f(x,s)
    =\sup_{\tht',\sig'}\bigl\{
    W^{\tht'\sig'}(x,s;x_0,s_0)+a^{\tht'\sig'}_{(x_0,s_0)}(f)\bigr\}
    =W^{\tht\sig}(x,s;x_0,s_0)=\h(x,s).
\end{align*}
Thus, $\h=g+a$, which proves that $\h$ is minimal. Then its extremility follows from Lemma \ref{eternal:min=ext}.
\end{proof}

Together, Theorems \ref{thm:decomp} and \ref{thm:min} show that any (normalized) eternal solution can be expressed as a convex mixsture of extremal ones. This establishes Theorem \ref{thm:main1}, except for the countability of the decomposition, which follows from parts \eqref{prop:general-I.a} and \eqref{thm:general-I.d} of Theorem \ref{thm:general-I} below (see Remark \ref{rk:countable}).\smallskip

Recall from the end of Section \ref{sec:main} that the (max-plus) Martin boundary is given by the horofunctions. For each $\tht\in\R$ and $\sigg\in\{-,+\}$, the function $(x,s)\mapsto W^{\tht\sig}(x,s;x_0,s_0)$ is a Busemann function and hence a horofunction, and therefore lies in the Martin boundary; by Theorem \ref{thm:min}, it belongs to the minimal Martin boundary. Conversely, if $\h$ lies in the minimal Martin boundary, then it is an extremal eternal solution, and Theorem \ref{thm:decomp} implies that $\h(\abullet)=W^{\tht\sig}(\,\abullet\,;x_0,s_0)$ for some $\tht\in\R$ and $\sigg\in\{-,+\}$. It follows that the minimal Martin boundary coincides with the set of Busemann functions. We now turn to identifying the full Martin boundary, that is, the set of horofunctions, beginning with the observation that any eternal solution with a spatial growth rate is necessarily a horofunction.

\begin{thm}\label{thm:eternal+growth->horo}
The following holds on a full probability event.
Let $(x_0,s_0)\in\R^2$. Let $\h$ be an eternal solution, normalized to vanish at $(x_0,s_0)$. Suppose that for some $r\in\R$,
\begin{align}\label{aux841}
\lim_{|z|\to\infty} z^{-1} \h(z,r)=2\tht,
\end{align}
for some $\tht\in\R$. Then there exist $a_-,a_+\in[-\infty,0]$ with $a_-\vee a_+=0$ and such that 
\begin{align}\label{h-rep2}
\h(x,s)=
 (W^{\tht-}(x,s;x_0,s_0)+a_-)\vee(W^{\tht+}(x,s;x_0,s_0)+a_+)
\end{align}
for all $x,s\in\R^2$.
\end{thm}

\begin{proof}
From \eqref{h-rep}, we have
\[\h(z,r)=\sup_{\tht',\sig'}\bigl\{W^{\tht'\sig'}(z,r;x_0,s_0)+a^{\tht'\sig'}_{(x_0,s_0)}(\h)\bigr\}.\]
From this, we get
\[\h(z,r)\ge W^{\tht'\sig'}(z,r;x_0,s_0)+a^{\tht'\sig'}_{(x_0,s_0)}(\h).\]
Suppose $a^{\tht'\sig'}_{(x_0,s_0)}(\h)>-\infty$.
Divide by $z$, take $z\to\infty$, and apply \eqref{Wgrowth} to get $\tht\ge\tht'$. Similarly, taking $z\to-\infty$ gives $\tht\le\tht'$. This shows that to have $a^{\tht'\sig'}_{(x_0,s_0)}(\h)>-\infty$, we must have $\tht'=\tht$. Now, the representation \eqref{h-rep} reduces to \eqref{h-rep2}.
\end{proof}

Next, we show the converse of Theorem \ref{thm:eternal+growth->horo}: any horofunction is an eternal solution with a spatial growth rate. 
To this end, we first show that eternal solutions of the form \eqref{h-rep2} must have a spatial growth rate of $2\tht$, at all times. Consequently, once \eqref{aux841} holds for some $r$, it, in fact, holds for all $r\in\R$.

\begin{lem}\label{lm:h-growth-rate}
    The following holds on a full probability event.
    Let $\tht\in\R$ and $a_-,a_+\in[-\infty,0]$ with $a_-\vee a_+=0$. Set
    \begin{align}\label{h-convcomb}
    \h(x,s)=(W^{\tht-}(x,s;x_0,s_0)+a_-)\vee(W^{\tht+}(x,s;x_0,s_0)+a_+)\quad\text{for all $x,s\in\R$.}
    \end{align}
    Then for all $x<y$ and $s$ in $\R$,
    \begin{align}\label{h-order}
    W^{\tht-}(y,s;x,s)\le \h(y,s)-\h(x,s)\le W^{\tht+}(y,s;x,s).
    \end{align}
    Furthermore, for all $s\in\R$
    \[\lim_{y-x\to\infty}(y-x)^{-1}(\h(y,s)-\h(x,s))=2\tht.\]
\end{lem}

\begin{rmk}
By \eqref{(5.6)}, for any $\tht\in\baddir$ and any $a_-,a_+\in\R$ with $a_-\vee a_+=0$, if $\h$ is given by \eqref{h-convcomb}, then
$\h(x,s_0)=W^{\tht+}(x,s_0;x_0,s_0)+a_+$ for all sufficiently large positive $x$, and $\h(x,s_0)=W^{\tht-}(x,s_0;x_0,s_0)+a_-$ for all sufficiently large negative $x$. Consequently, distinct pairs $(a_-,a_+)$ determine distinct eternal solutions \eqref{h-convcomb}. In particular, there exists a continuum of eternal solutions with growth rate $2\tht$.
\end{rmk}

\begin{proof}[Proof of Lemma \ref{lm:h-growth-rate}]
Use \eqref{cocycle} and \eqref{W:mono} to write
\begin{align*}
 &\h(y,s)-\h(x,s)\\
 &\quad=(W^{\tht-}(y,s;x_0,s_0)+a_-)\vee(W^{\tht+}(y,s;x_0,s_0)+a_+)\\
 &\qquad\qquad\qquad-(W^{\tht-}(x,s;x_0,s_0)+a_-)\vee(W^{\tht+}(x,s;x_0,s_0)+a_+)\\
 &\quad=(W^{\tht-}(y,s;x,s)+W^{\tht-}(x,s;x_0,s_0)+a_-)\vee(W^{\tht+}(y,s;x,s)+W^{\tht+}(x,s;x_0,s_0)+a_+)\\
 &\qquad\qquad\qquad-(W^{\tht-}(x,s;x_0,s_0)+a_-)\vee(W^{\tht+}(x,s;x_0,s_0)+a_+)\\
 &\quad\le(W^{\tht+}(y,s;x,s)+W^{\tht-}(x,s;x_0,s_0)+a_-)\vee(W^{\tht+}(y,s;x,s)+W^{\tht+}(x,s;x_0,s_0)+a_+)\\
 &\qquad\qquad\qquad-(W^{\tht-}(x,s;x_0,s_0)+a_-)\vee(W^{\tht+}(x,s;x_0,s_0)+a_+)\\
&\quad=W^{\tht+}(y,s;x,s).
\end{align*}
The other inequality comes similarly.
The limit claim comes from \eqref{h-order} and  \eqref{Wgrowth}.
\end{proof}

 Next, we show that every horofunction is a max-plus convex combination of the form \eqref{h-rep2}, for some $\tht\in\R$. In particular, by Lemma \ref{eternal-cvx}, this shows that horofunctions are eternal solutions.

\begin{thm}\label{Buslim}
The following holds on a full probability event.
Let $x_0,s_0\in\R$. Let $(r_n,z_n)\in\R^2$ be a sequence such that $r_n\to\infty$ and $\mathcal L(\,\abullet\,;z_n,r_n)-\mathcal L(x_0,s_0;z_n,r_n)$ converges uniformly on compacts to a function $\h:\R^2\to\R$. Then the following hold.
\begin{enumerate} [label={\rm(\alph*)}, ref={\rm\alph*}]   \itemsep=3pt 
\item\label{Buslim.a} There exists a $\tht\in\R$ such that $z_n/r_n\to\tht$ as $n\to\infty$. 
\item\label{Buslim.b} $\h$ is an eternal solution, normalized to satisfy $\h(x_0,s_0)=0$.
\item\label{Buslim.c} There exist $a_-,a_+\in\R\cup\{-\infty\}$ such that $a_-\vee a_+=0$ and 
\[\h(x,s)=(W^{\tht-}(x,s;x_0,s_0)+a_-)\vee(W^{\tht+}(x,s;x_0,s_0)+a_+)\quad\text{for all $x,s\in\R$.}\]
\end{enumerate}
\end{thm}

\begin{proof}
Part \eqref{Buslim.a}. Suppose first that there exists a subsequence $n_k$ such that $z_{n_k}/r_{n_k}\to\infty$. Take any rational $x$ and $s$ and a rational $\tht>0$. 
By \cite[Theorem 5.5(iii)]{Bus-Sep-Sor-24}, we have $\tht \notin \baddir$. Consequently, there is no sign distinction, and we simply write $W^\tht$.
For $k$ sufficiently large, $\tht r_{n_k}<z_{n_k}$ and, by \cite[Lemma B.4(i)]{Bus-Sep-Sor-24}, 
\[\mathcal L(x+1,s;\tht r_{n_k},r_{n_k})-\mathcal L(x,s;\tht r_{n_k},r_{n_k})\le\mathcal L(x+1,s;z_{n_k},r_{n_k})-\mathcal L(x,s;z_{n_k},r_{n_k}).\]
Taking $k\to\infty$ and applying 
\cite[Theorem 5.1(vi)]{Bus-Sep-Sor-24} gives 
\[W^{\tht}(x+1,s;x,s)\le \h(x+1,s)-\h(x,s).\]
By \eqref{W:mono}, the left-hand side is nondecreasing in $\tht$. Stationarity and \eqref{Wgrowth} (or see \cite[Theorem 5.3(iii)]{Bus-Sep-Sor-24}) imply that $\mathbb E[W^{\tht}(x+1,s;x,s)]=2\tht$. By monotone convergence, $W^{\tht}(x+1,s,x,s)\nearrow\infty$ as $\tht\nearrow\infty$. This implies $\h(x+1,s)-\h(x,s)=\infty$, which contradicts the finiteness of $\h$.
Thus, it must be that $z_n/r_n$ is bounded above. A similar reasoning proves that it is also bounded below.  Compactness then implies the existence of a subsequence $n_k$ and a $\tht\in\R$ such that $z_{n_k}/r_{n_k}\to\tht$. Then \cite[Theorem 5.1(vii)]{Bus-Sep-Sor-24} implies that for any $x<y$ and $s$ in $\R$, 
\begin{align}\label{W<h<W}
W^{\tht-}(y,s;x,s)\leq \h(y,s)-\h(x,s)\leq W^{\tht+}(y,s;x,s).
\end{align}

Suppose there exists another $\tht'\in\R$ and a different subsequence $n'_k$ such that $z_{n'_k}/r_{n'_k}\to\tht'$. 
Assume $\tht>\tht'$, the case $\tht<\tht'$ being treated similarly.
The above argument, with $s=s_0$, $x=x_0$, and $\tht'$ in place of $\tht$ gives
\begin{align}\label{W<h<W'}
W^{\tht'-}(y,s_0;x_0,s_0)\leq \h(y,s_0)\leq W^{\tht'+}(y,s_0;x_0,s_0)
\end{align}
for all $y>x_0$. 
 Then the monotonicity \eqref{W:mono} implies $W^{\tht'+}(y,s_0;x_0,s_0)\le W^{\tht-}(y,s_0;x_0,s_0)$.
This, \eqref{W<h<W'}, and \eqref{W<h<W} give $W^{\tht'+}(y,s_0;x_0,s_0)=W^{\tht-}(y,s_0;x_0,s_0)$. Dividing by $y$, taking $y\to\infty$, and using \eqref{Wgrowth}, we have $\tht'=\tht$, contradicting the assumption that $\tht'\ne\tht$. Thus, it must be that $z_n/r_n$ converges to a limit $\tht\in\R$.\smallskip

Part \eqref{Buslim.b}. Take $s<t$ and $x$ in $\R$. By the superadditivity \eqref{superadditivity}, we have 
\[\mathcal L(x,s;y,t)+\mathcal L(y,t;z_n,r_n)\le
\mathcal L(x,s;z_n,r_n),\]
for all $y\in\R$ and all $n$ sufficiently large (so that $r_n>t$). Subtract $\mathcal L(x_0,s_0;z_n,r_n)$ from both sides and take $n\to\infty$ to get
\[\mathcal L(x,s;y,t)+\h(y,t)\le \h(x,s).\]
Take a supremum over $y$ to get
\begin{align}\label{aux435}
\sup_y\bigl\{\mathcal L(x,s;y,t)+\h(y,t)\bigr\}\le \h(x,s).
\end{align}
To conclude that $\h$ is an eternal solution, we prove the reverse inequality.

For each $n$ (with $r_n>t$) let $\gamma_n:[s,r_n]\to\R$ be the leftmost geodesic from $(x,s)$ to $(z_n,r_n)$. Let $y_n=\gamma_n(t)$. Thus,
\begin{align}\label{aux462}
\mathcal L(x,s;y_n,t)+\mathcal L(y_n,t;z_n,r_n)=\mathcal L(x,s;z_n,r_n).
\end{align}
Let $a\le b$ be such that $\geo\from{(x,s)}\dir{L}{\tht}{-}(t)=(a,t)$ and $\geo\from{(x,s)}\dir{R}{\tht}{+}(t)=(b,t)$.
Then \cite[Theorem 6.5(i)]{Bus-Sep-Sor-24} says that 
\[a\le\varliminf_{n\to\infty}y_n\le\varlimsup_{n\to\infty}y_n\le b.\]
By compactness, there exists a subsequence $n_k$ and a $y\in[a,b]$ such that $y_{n_k}\to y$.
Using \eqref{aux462},
\begin{align*}
&\bigl|\mathcal L(x,s;y,t)+\mathcal L(y,t;z_{n_k},r_{n_k})-\mathcal L(x,s;z_{n_k},r_{n_k})\bigr|\\
&\qquad\le\bigl|\mathcal L(x,s;y,t)-\mathcal L(x,s;y_{n_k},t)\bigr|
+\bigl|\mathcal L(y,t;z_{n_k},r_{n_k})-\mathcal L(y_{n_k},t;z_{n_k},r_{n_k})\bigr|\\
&\qquad=\bigl|\mathcal L(x,s;y,t)-\mathcal L(x,s;y_{n_k},t)\bigr|\\
&\qquad\qquad+\bigl|\bigl(\mathcal L(y,t;z_{n_k},r_{n_k})-\mathcal L(x_0,s_0;z_{n_k},r_{n_k})\bigr)-\bigl(\mathcal L(y_{n_k},t;z_{n_k},r_{n_k})-\mathcal L(x_0,s_0;z_{n_k},r_{n_k})\bigr)\bigr|.
\end{align*}
By the continuity of $\mathcal L$ and the uniform convergence on compacts of $\mathcal L(\,\abullet\,;z_n,r_n)-\mathcal L(x_0,s_0;z_n,r_n)$ to $\h(\abullet)$, the right-hand side converges to $0$ as $k\to\infty$. Then adding and subtracting $\mathcal L(x_0,s_0;z_{n_k},r_{n_k})$ inside the absolute value on the left-hand side, then taking $k\to\infty$ gives 
\[\mathcal L(x,s;y,t)+\h(y,t)=\h(x,s).\]
This and \eqref{aux435} imply \eqref{h-eternal} and hence $\h$ has been shown to be an eternal solution. It is clear from the definition of $\h$ that $\h(x_0,s_0)=0$.
\smallskip

Part \eqref{Buslim.c}. From \eqref{h-rep}, we have
\[\h(x,s_0)=\sup_{\tht',\sig'}\bigl\{W^{\tht'\sig'}(x,s_0;x_0,s_0)+a^{\tht'\sig'}_{(x_0,s_0)}(\h)\bigr\}.\]
This and \eqref{W<h<W} (with $y=x_0$ and $s=s_0$) give
\begin{align}\label{aux988}
W^{\tht'\sig'}(x,s_0;x_0,s_0)+a^{\tht'\sig'}_{(x_0,s_0)}(\h)\le W^{\tht-}(x,s_0;x_0,s_0)\quad\text{for all $x<x_0$}
\end{align}
and (applying \eqref{W<h<W} with $s=s_0$, $x=x_0$, and using $x$ in place of $y$)
\begin{align}\label{aux992}
W^{\tht'\sig'}(x,s_0;x_0,s_0)+a^{\tht'\sig'}_{(x_0,s_0)}(\h)\le W^{\tht+}(x,s_0;x_0,s_0)\quad\text{for all $x>x_0$}.
\end{align}

Suppose $a^{\tht'\sig'}_{(x_0,s_0)}(\h)>-\infty$. Divide both sides of \eqref{aux988} by $x$, take $x\to-\infty$, and apply \eqref{Wgrowth} to get $\tht'\ge\tht$. Dividing \eqref{aux992} by $x$ and taking $x\to\infty$ gives $\tht'\le\tht$. Thus, the only $\tht'$ for which $a^{\tht'\sig'}_{(x_0,s_0)}(\h)>-\infty$ is $\tht'=\tht$. Now, \eqref{h-rep} reduces to the claim in part \eqref{Buslim.c}.
\end{proof}

Thus far, Theorem \ref{Buslim} showed that horofunctions are eternal solutions that are max-plus convex combinations of the Busemann functions $W^{\tht\pm}$ for some $\tht\in\R$, while Theorem \ref{thm:eternal+growth->horo} and Lemma \ref{lm:h-growth-rate}, identify these with the eternal solutions having spatial growth rate $2\tht$. Theorem \ref{thm:convcomb} below completes the equivalence by showing that such convex combinations are horofunctions. To prove this, we need a more detailed description of combinations of the functions $W^{\tht\pm}$.\smallskip

For $\tht\in\baddir$, $(x_0,s_0)\in\R^2$, and $a\in\R$, let  
\begin{align}\label{I}
\begin{split}
  I_{(x_0,s_0),a}^{\tht+}(s)&=\inf\{x\in\R:W^{\tht-}(x,s;x_0,s_0)-W^{\tht+}(x,s;x_0,s_0)<a\}\quad\text{and}\\ 
  I_{(x_0,s_0),a}^{\tht-}(s)&=\sup\{x\in\R:W^{\tht-}(x,s;x_0,s_0)-W^{\tht+}(x,s;x_0,s_0)>a\},\quad s\in\R.
\end{split}
\end{align}

\begin{prop}\label{I-prop}
The following holds on a full probability event.
Let $\tht\in\baddir$ and $(x_0,s_0)\in\R^2$.  Let $a_-,a_+\in(-\infty,0]$ be such that $a_-\vee a_+=0$. 
Let
\[\h(x,s)=
 (W^{\tht-}(x,s;x_0,s_0)+a_-)\vee(W^{\tht+}(x,s;x_0,s_0)+a_+),\quad x,s\in\R.\]
Let $a=a_+-a_-$ and take $s\in\R$.
\begin{enumerate} [label={\rm(\alph*)}, ref={\rm\alph*}] \itemsep=1pt 
\item\label{I-prop.a} $-\infty<I_{(x_0,s_0),a}^{\tht-}(s)\le I_{(x_0,s_0),a}^{\tht+}(s)<\infty$.
\item\label{I-prop.b} $W^{\tht-}(x,s;x_0,s_0)+a_-<W^{\tht+}(x,s;x_0,s_0)+a_+=\h(x,s)$ for all $x>I_{(x_0,s_0),a}^{\tht+}(s)$.
\item\label{I-prop.c} $\h(x,s)=W^{\tht-}(x,s;x_0,s_0)+a_->W^{\tht+}(x,s;x_0,s_0)+a_+$ for all $x<I_{(x_0,s_0),a}^{\tht-}(s)$.
\item\label{I-prop.d} $W^{\tht-}(x,s;x_0,s_0)+a_-=W^{\tht+}(x,s;x_0,s_0)+a_+=\h(x,s)$ for $x\in[I_{(x_0,s_0),a}^{\tht-}(s),I_{(x_0,s_0),a}^{\tht+}(s)]$.
\end{enumerate}
\end{prop}

\begin{proof}
\eqref{(5.6)} and the continuity of $W^{\tht-}$ and $W^{\tht+}$ imply  $I_{(x_0,s_0),a}^{\tht-}(s)>-\infty$ and $I_{(x_0,s_0),a}^{\tht+}(s)<\infty$.
By the cocycle property \eqref{cocycle} and the monotonicity \eqref{W:mono}, if $z>x$, then
\begin{align*}
&W^{\tht-}(z,s;x_0,s_0)-W^{\tht+}(z,s;x_0,s_0)\\
&\qquad=\bigl(W^{\tht-}(x,s;x_0,s_0)-W^{\tht+}(x,s;x_0,s_0)\bigr)
-\bigl(W^{\tht-}(x,s;z,s)-W^{\tht+}(x,s;z,s)\bigr)\\
&\qquad\le
W^{\tht-}(x,s;x_0,s_0)-W^{\tht+}(x,s;x_0,s_0).
\end{align*}
Hence, the function $x \mapsto W^{\tht-}(x,s;x_0,s_0)-W^{\tht+}(x,s;x_0,s_0)$ is nonincreasing. This implies that $I_{(x_0,s_0),a}^{\tht-}(s)\le I_{(x_0,s_0),a}^{\tht+}(s)$, establishing part \eqref{I-prop.a}. The  nonincreasing behavior also yields parts \eqref{I-prop.b} and \eqref{I-prop.c}. Moreover, it follows that $W^{\tht-}(x,s;x_0,s_0)+a_-=W^{\tht+}(x,s;x_0,s_0)+a_+$ for all $x$ strictly between $I_{(x_0,s_0)}^{\tht-}(s)$ and $I_{(x_0,s_0)}^{\tht+}(s)$.

By the continuity of $W^{\tht-}$ and $W^{\tht+}$, we get that for both $\sigg\in\{-,+\}$,
\begin{align}\label{aux912}
W^{\tht-}(I_{(x_0,s_0),a}^{\tht\sig}(s),s;x_0,s_0)+a_-=W^{\tht+}(I_{(x_0,s_0),a}^{\tht\sig}(s),s;x_0,s_0)+a_+.
\end{align}
We have thus shown part \eqref{I-prop.d} to hold.
\end{proof}

Now we give the converse of Theorem \ref{Buslim}: any convex combination of $W^{\tht\pm}$ is a horofunction. 

\begin{thm}\label{thm:convcomb}
The following holds on a full probability event.
Let  $x_0,s_0,\tht\in\R$. 
Let $a_-,a_+\in\R\cup\{-\infty\}$ be such that $a_-\vee a_+=0$. Let $\h(x,s)=(W^{\tht-}(x,s;x_0,s_0)+a_-)\vee(W^{\tht+}(x,s;x_0,s_0)+a_+)$, $x,s\in\R$. Then there exists a sequence $(z_n,r_n)\in\R^2$ with $r_n\to\infty$, $z_n/r_n\to\tht$, such that, as $n\to\infty$, 
$\mathcal L(\,\abullet\,;z_n,r_n)-\mathcal L(x_0,s_0;z_n,r_n)$ converges to $\h$, uniformly on compact sets.
\end{thm}

\begin{rmk}
 The convergence in Theorem \ref{thm:convcomb} for compact sets contained in a single time slice $\R\times\{t\}$, $t\in\R$, follows from \cite[Theorem 1.14]{Bha-Bus-Sor-25-}. The main additional difficulty is extending this convergence to arbitrary compact subsets of $\R^2$.
\end{rmk}

\begin{proof}[Proof of Theorem \ref{thm:convcomb}]
If $a_-=-\infty$, then $\h(x,s)=W^{\tht+}(x,s;x_0,s_0)$, and the claim follows from the coalescence  \eqref{geo:coal2}, the definition \eqref{W-def}, and the $\tht$-directedness of $W^{\tht+}$-geodesics in \cite[Theorem 5.9(ii)(d)]{Bus-Sep-Sor-24}. The same argument applies when $a_+=-\infty$, as well as in the case $\tht\notin\baddir$ (where $W^{\tht-}=W^{\tht+}$). Hence, we may assume that $a_-,a_+\in\R$ and $\tht\in\baddir$. 

 We first work with a fixed compact set $K\subset\R^2$ and construct $z_r(K)\in\R$ and $T_K''<\infty$ such that $z_r(K)/r\to\tht$ as $r\to\infty$ and
 \begin{align}\label{goal1}
 \h(x,s)-\h(y,t)=\LL(x,s;z_r(K),r)-\LL(y,t;z_r(K),r),\quad\forall(x,s),(y,t)\in K
 \text{ and }\forall r\ge T_K''.
 \end{align}
Once this is achieved, we will cover $\R^2$ with an increasing sequence of compact sets and extract the sequence $(z_n,r_n)$ as in the statement of the theorem.\smallskip

Let  $ I^- = I^{\tht-}_{(x_0,s_0),a} $, defined in \eqref{I}. Abbreviate $p=I^-(s_0)$.
Let $ K\subset \R^2 $ be a compact such that $p\in K$.  
Let $\gamma^-=\geo\from{p}^{L,\tht-}$ and $\gamma^+=\geo\from{p}^{R,\tht+}$.
By \cite[Theorem 7.1(iii)]{Bus-Sep-Sor-24}, there exists $ T_K > \sup\{t:(y,t)\in K\}$ such that, for both $\sigg\in\{-,+\}$, all the $W^{\tht\sig}$-Busemann geodesics out of points in $K$ have coalesced with $\gamma^\sig$ by time $T_K$.  
Let $K'$ be the closure of
\[\bigcup_{(x,s)\in K} \Bigl\{\geo\from{(x,s)}^{S,\tht\sig}([s,T_K]):\sigg\in\{-,+\},S\in\{L,M,R\}\Bigr\},\]
the set of all semi-infinite geodesics, from points in $K$, running up to time $T_K$.
By Lemma 5.13 in \cite{Bus-Sep-Sor-24}, $K'$ is compact. Then, Theorem 7.1(iii) in \cite{Bus-Sep-Sor-24} again gives a time $ T_K' > T_K $ by which, for $\sigg\in\{-,+\}$, all $W^{\tht\sig}$-Busemann geodesics out of $ K' $ have coalesced with $\gamma^\sig$. By \eqref{sign}, we can choose $T_K'$ even larger, if necessary, to also have $\gamma^-(T_K')<\gamma^+(T_K')$.

    Given $x,s,y,t\in\R$ with $t>s$, let $\geo_{(x,s):(y,t)}^S$, $S\in\{L,R\}$, denote, respectively, the leftmost and rightmost geodesics from $(x,s)$ to $(y,t)$. These paths exist by \cite[Lemma 13.2]{Dau-Ort-Vir-22}. 
        
		Next, for $ r\geq T_K' $, define 
		\begin{align*}
			z_r(K)= \sup\bigl\{z:\gamma^-(r)\le (z,r)\le\gamma^+(r) \, : \, \geo^L_{p:(z,r)}(T_K') = \gamma^-(T_K')  \bigr\},
		\end{align*}
		and write $ v_r(K) = (z_r(K),r) $. Since $K$ is fixed for the moment, we omit the dependence on it from the notation.
        The set in the supremum is nonempty since $ \gamma^- $ is the leftmost geodesic between any two of its points and hence $z$ with $(z,r)=\gamma^-(r)$ belongs to the set. The $\tht$-directedness of $\gamma^-$ and $\gamma^+$ in \cite[Theorem 5.9(ii)(d)]{Bus-Sep-Sor-24} implies that $z_r/r\to\tht$ as $r\to\infty$.

 By Lemma \ref{lm9.1}, there exists $T_K'' > T_K'$ such that for all $ r \geq T_K'' $ and $ (x,s) \in K' $,
		\begin{align}
			&\geo^L_{(x,s):v_r}\big|_{[s,T_K']} \in\Bigl\{\geo^{L,\tht-}_{(x,s)}\big|_{[s,T_K']}\,,\,\geo^{L,\tht+}_{(x,s)}\big|_{[s,T_K']}\Bigr\}
            \quad\text{and}\label{aux:ef9.1L}\\[5pt]
			&\geo^R_{(x,s):v_r}\big|_{[s,T_K']}
            \in
            \Bigl\{\geo^{R,\tht-}_{(x,s)}\big|_{[s,T_K']}\,,\,\geo^{R,\tht+}_{(x,s)}\big|_{[s,T_K']}\Bigr\}.
            \label{aux:ef9.1R}
		\end{align}

        By planarity and the continuity of $\LL$,  $\geo_{p:(x,r)}^L$ converges to $\geo_{p:(z,r)}^L$ as $x\nearrow z$. Thus, the supremum in the definition of $z_r$ is attained and
        \begin{align}\label{aux1116}
        \geo_{p:v_r}^L(T'_K)=\gamma^-(T'_K).
        \end{align}
        Similarly, by the continuity of $\LL$, as $z\searrow z_r$, $\geo_{p:(z,r)}^L$ converges to a geodesic $\tau$ from $p$ to $v_r$. Planarity and the leftmost property imply that these paths are nonincreasing and then Dini's theorem \cite[Theorem 7.13]{Rud-76} says the convergence is uniform. Theorem 1.18 in \cite{Bat-Gan-Ham-22} then implies that $\geo_{p:(z,r)}^L(T'_K)=\tau(T'_K)$, once $z$ is close enough to $z_r$. But the maximality of $z_r$ implies that $\geo_{p:(z,r)}^L(T_K')>\gamma^-(T_K')$ for all $z>z_r$. Therefore, $\gamma^-(T'_K)<\tau(T'_K)\le\geo_{p:v_r}^R(T_K')$ and, by \eqref{aux:ef9.1R},
        \begin{align}\label{aux1122}
        \geo^R_{p:v_r}(T'_K)=\gamma^+(T'_K).
        \end{align}
        
       We claim that for any $(x,s)\in K$, 
    \begin{align}\label{claim1114}
    x\le I^-(s)\Longrightarrow\geo^L_{(x,s):v_r}(T_K')=\gamma^-(T_K')\quad\text{and}\quad
    x\ge I^-(s)\Longrightarrow \geo^R_{(x,s):v_r}(T_K')=\gamma^+(T_K').
    \end{align}
    We prove the second implication, the first being similar. Let $ (x,s) \in K $ such that $  x\ge I^-(s)$.  By \eqref{aux:ef9.1R} and the definition of $ T_K' $, $ \geo^R_{(x,s):v_r}(T_K')\in \{\gamma^-(T_K'),\gamma^+(T_K')\} $, so suppose for contradiction that 
    \begin{align}\label{aux1123}
    \geo^R_{(x,s):v_r}(T_K') = \gamma^-(T_K').
    \end{align}
		
		By \eqref{aux1116}-\eqref{aux1122} and $\gamma^\sig$ being a $W^{\tht\sig}$-geodesic, $\sigg\in\{-,+\}$,
		\begin{align*}
        W^{\tht-}(p;\gamma^-(T_K')) + \LL(\gamma^-(T_K');v_r)
            &= \LL(p;\gamma^-(T_K')) + \LL(\gamma^-(T_K');v_r) = \LL(p;v_r) \\
			= \LL(p;\gamma^+(T_K')) + \LL(\gamma^+(T_K');v_r) &= W^{\tht+}(p;\gamma^+(T_K'))+\LL(\gamma^+(T_K');v_r).
		\end{align*}
		Rearranging yields
		\begin{align}\label{aux:subs}
			\LL(\gamma^-(T_K');v_r) = W^{\tht+}(p;\gamma^+(T_K')) - W^{\tht-}(p;\gamma^-(T_K'))+\LL(\gamma^+(T_K');v_r) .
		\end{align}
		
		Since $ x\ge I^-(s) $, parts \eqref{I-prop.b} and \eqref{I-prop.d} of Proposition \ref{I-prop}, together with \eqref{aux912} and the cocycle property \eqref{cocycle}, give $ W^{\tht-}(x,s;p) \leq W^{\tht+}(x,s;p) $. On the other hand, since $\gamma^-$ is a $W^{\tht-}$-geodesic, $W^{\tht-}(p;\gamma^-(T_K))=\LL(p;\gamma^-(T_K))$ and, using \eqref{W-eternal}, we have $\LL(p;\gamma^-(T_K))\le W^{\tht+}(p;\gamma^-(T_K))$. This and $\geo_{(x,s)}^{L,\tht-}(T_K)=\gamma^-(T_K)$ give us $W^{\tht-}(\geo_{(x,s)}^{L,\tht-}(T_K);p) \ge W^{\tht+}(\geo_{(x,s)}^{L,\tht-}(T_K);p) $. By the continuity of $\geo_{(x,s)}^{L,\tht-}$, $W^{\tht-}$, and $W^{\tht+}$, there exists $ t' \in [s,T_K] $  such that, setting $ q=\geo^{L,\tht-}_{(x,s)}(t')$, we have 
        \begin{align}\label{aux1139}
        W^{\tht-}(q;p) = W^{\tht+}(q;p).
        \end{align}
        Now, by \eqref{aux1123}, $\gamma^-(T'_K)$ is on a geodesic from $(x,s)$ to $v_r$ and, by the definition of $T_K$, it is also on $\geo^{L,\tht-}_{(x,s)}$. From this, we get two conclusions: $q$ is on a geodesic from $(x,s)$ to $\gamma^-(T'_K)$ and $ \LL(q;\gamma^-(T'_K)) = W^{\tht-}(q;\gamma^-(T'_K)) $. This justifies, respectively, the first and third equalities below. Since $q\in K'$, the definition of $T'_K$ ensures that the $W^{\tht+}$-geodesics from $q$ go through $\gamma^+(T'_K)$. This justifies the last equality. Thus,
		\begin{align*}
			\LL(x,s;v_r) &= \LL(x,s;q) + \LL(q;\gamma^-(T_K')) + \LL(\gamma^-(T_K');v_r)\\
			&\!\!\!\overset{\eqref{aux:subs}}= \LL(x,s;q) + \LL(q;\gamma^-(T_K')) +W^{\tht+}(p;\gamma^+(T_K'))- W^{\tht-}(p;\gamma^-(T_K'))+ \LL(\gamma^+(T_K');v_r)\\
			&= \LL(x,s;q) + W^{\tht-}(q;\gamma^-(T_K'))+W^{\tht+}(p;\gamma^+(T_K')) - W^{\tht-}(p;\gamma^-(T_K')) + \LL(\gamma^+(T_K');v_r)\\
			&\!\!\!\overset{\eqref{cocycle}}= \LL(x,s;q) +W^{\tht+}(p;\gamma^+(T_K')) + W^{\tht-}(q;p)+ \LL(\gamma^+(T_K');v_r)\\
			&\!\!\!\overset{\eqref{aux1139}}= \LL(x,s;q) +W^{\tht+}(p;\gamma^+(T_K'))+ W^{\tht+}(q;p)+ \LL(\gamma^+(T_K');v_r)\\
			&\!\!\!\overset{\eqref{cocycle}}= \LL(x,s;q) +W^{\tht+}(q;\gamma^+(T_K'))+ \LL(\gamma^+(T_K');v_r)\\
			&= \LL(x,s;q) + \LL(q;\gamma^+(T_K'))+ \LL(\gamma^+(T_K');v_r).
		\end{align*}
		This places $ \gamma^+(T_K') $ on a geodesic from $ (x,s) $ to $ v_r $, contradicting \eqref{aux1123}. We have proved \eqref{claim1114}. 
		
		Now let $ (x,s), (y,t) \in K$ such that $ (x,s),(y,t) \preceq I^- $. Then, by \eqref{claim1114} and the definition of $T'_K$,
		\begin{align*}
			\LL(x,s;v_r)-\LL(y,t;v_r) &= [\LL(x,s;\gamma^-(T_K'))+\LL(\gamma^-(T_K');v_r)] - [\LL(y,t;\gamma^-(T_K'))+\LL(\gamma^-(T_K');v_r)]\\
			&=\LL(x,s;\gamma^-(T_K')) -\LL(y,t;\gamma^-(T_K'))\\
			&= W^{\tht-}(x,s;\gamma^-(T_K')) -W^{\tht-}(y,t;\gamma^-(T_K'))\\
			&= W^{\tht-}(x,s;y,t).
		\end{align*}
		
		When instead $ (x,s),(y,t) \succeq I^- $, the same argument gives
		\begin{align*}
			\LL(x,s;v_r)-\LL(y,t;v_r) = W^{\tht+}(x,s;y,t).
		\end{align*}

		In the remaining case $ (x,s) \preceq I^- \preceq (y,t) $, inserting $p \in I^-$ as an intermediate point gives
		\begin{align*}
			\LL(x,s;v_r)-\LL(y,t;v_r) &=\LL(x,s;v_r)-\LL(p;v_r)+\LL(p;v_r)-\LL(y,t;v_r)\\
			&= W^{\tht-}(x,s;p) + W^{\tht+}(p;y,t).
		\end{align*}
		On the other hand, by Proposition \ref{I-prop}\eqref{I-prop.b}-\eqref{I-prop.d},
		\begin{align*}
			\h(x,s)-\h(y,t) =\begin{cases} 
            W^{\tht-}(x,s;y,t) &\text{if } (x,s),(y,t)\preceq I^-,\\
			W^{\tht+}(x,s;y,t) &\text{if } (x,s),(y,t)\succeq I^-,\\
			W^{\tht-}(x,s;p) + W^{\tht+}(p;y,t) &\text{if } (x,s) \preceq I^- \preceq (y,t).
            \end{cases}
		\end{align*}
        This proves \eqref{goal1}. We now prove the statement of the theorem.

        Let $K_n\subset\R^2$, $n\in\N$, be an increasing sequence of compact sets that exhausts $\R^2$ and such that $(x_0,s_0)$ and $I^-(s_0)$ are in $K_1$. Define $r_n=\max_{1\le m\le n}T''_{K_m}$
        and abbreviate $z_n=z_{r_n}(K_n)$. Then for any $n\ge m$ in $\N$, $r_n\ge T''_{K_m}$ and, by \eqref{goal1}, with $(y,t)=(x_0,s_0)$,
        \[\h(x,s)=\LL(x,s;z_n,r_n)-\LL(x_0,s_0;z_n,r_n).\]
        This implies the uniform convergence, on $K_m$, of $\LL(\abullet;z_n,r_n)-\LL(x_0,s_0;z_n,r_n)$ to $\h$. Since this holds for any $m\in\N$, the theorem is proved.
    \end{proof}

Putting together Lemma \ref{lm:h-growth-rate} and Theorems \ref{thm:eternal+growth->horo}, \ref{Buslim}, and \ref{thm:convcomb} proves Theorem \ref{thm:main2}.\smallskip

We now give a more detailed description of the formula \eqref{h-rep}. 
Equip $\R\times\{-,+\}$ with the lexicographic order: $(\tht_1,\sigg_1)\preceq(\tht_2,\sigg_2)$ if $\tht_1<\tht_2$, or $\tht_1=\tht_2$ and $\sigg_1\le\sigg_2$, where $-\le+$. For $\tht\notin\baddir$, we identify $(\tht,-)=(\tht,+)$, while for $\tht\in\baddir$ these two are distinct.

Since \eqref{geodesics} exhausts all semi-infinite geodesics, for any eternal solution $\h$, point $(x,s)\in\R^2$, and $S\in\{L,R\}$, there exists a unique pair $\vartheta^{S,\h}_{(x,s)}=(\tht,\sigg)$ such that $\geo\from{(x,s)}^{S,\h}$ is a $W^{\tht\sig}$-geodesic. By \eqref{geo:coal2}, this geodesic coalesces with $\geo\from{(z,r)}^{S,\tht\sig}$ for all $(z,r)\in\R^2$.
Let
    \begin{align}\label{def:jumps}
    \jumps^{S,\h}_s=\bigl\{ \vartheta^{S,\h}_{(x,s)} : x\in\R \bigr\}.
    \end{align}
Note that $\jumps^{S,\h}_s$ is never empty.
For $\vartheta\in\R\times\{-,+\}$, define
\begin{align}\label{def:I-general}
I^{L,\h}_{\vartheta}(s)=\sup\bigl\{x:\vartheta^{L,\h}_{(x,s)}\preceq\vartheta\bigr\}\in[-\infty,\infty]
\quad\text{and}\quad
I^{R,\h}_{\vartheta}(s)=\inf\bigl\{x:\vartheta^{R,\h}_{(x,s)}\succeq\vartheta\bigr\}\in[-\infty,\infty],
\end{align}
with the convention that $\sup\emptyset=-\infty$ and $\inf\emptyset=\infty$. These mappings generalize the interfaces defined in \eqref{I}. The next proposition records some of their properties.

\begin{prop}\label{prop:general-I}
There exists an event of full probability on which the following hold, for any eternal solution $\h$ and any $s\in\R$.
     \begin{enumerate} [label={\rm(\alph*)}, ref={\rm\alph*}]   \itemsep=3pt 
      \item\label{prop:general-I.a} For each $S \in \{L,R\}$, $\mathcal V^{S,\h}_r\subset\mathcal V^{S,\h}_s$ for all $r\le s$. 
     \item\label{prop:general-I.b} For all $x\in\R$, $\vartheta^{L,\h}_{(x,s)}\preceq\vartheta^{R,\h}_{(x,s)}$.
     \item\label{prop:general-I.c} For each $S \in \{L,R\}$, the map $x \mapsto \vartheta^{S,\h}_{(x,s)}$ is nondecreasing. Consequently, for $S=L$ {\rm(}respectively, $S=R${\rm)}, the map $\vartheta \mapsto I^{S,\h}_\vartheta(s)$ is the corresponding right-continuous {\rm(}respectively, left-continuous{\rm)} generalized inverse.  
     \item\label{prop:general-I.d} For all $\vartheta_1\preceq\vartheta_2$ in $\R\times\{-,+\}$ and each $S\in\{L,R\}$, $I^{S,\h}_{\vartheta_1}(s)\le I^{S,\h}_{\vartheta_2}(s)$. 
     \item\label{prop:general-I.e} For any $x\in\R$, there exists $\e>0$ such that $\vartheta^{L,\h}_{(\abullet,s)}$ is constant on both $(x-\e,x]$ and $(x,x+\e)$. Consequently, the supremum in the definition of $I^{L,\h}_\vartheta(s)$ is always attained when it is finite.
     \item\label{prop:general-I.f} For any $x\in\R$, there exists $\e>0$ such that $\vartheta^{R,\h}_{(\abullet,s)}$ is constant on both $[x,x+\e)$ and $(x-\e,x)$. Consequently, the infimum in the definition of $I^{R,\h}_\vartheta(s)$ is always attained when it is finite.
     \item\label{prop:general-I.g} For $S=L$ {\rm(}respectively, $S=R${\rm)}, $x \mapsto \vartheta^{S,\h}_{(x,s)}$ is the left-continuous {\rm(}respectively, right-continuous{\rm)} generalized inverse of $\vartheta \mapsto I^{S,\h}_{\vartheta}(s)$. Consequently, the set of jump points of $I^{S,\h}_{\abullet}(s)$ is given by $\jumps^{S,\h}_s$ and the set of jump points of $\vartheta^{S,\h}_{(\abullet,s)}$ is given by 
     \[\Ijumps^{S,\h}_s=\{I^{S,\h}_{\vartheta}(s):\vartheta\in\R\times\{-,+\}\}\cap\R.\]
     \item\label{prop:general-I.h} For every $\tht\in\R$ for which $I^{L,\h}_{(\tht,+)}(s)$ is finite, there exists $\e>0$ such that $I^{L,\h}_{(\tht',\sig')}(s)=I^{L,\h}_{(\tht,+)}(s)$, for all $\tht'\in(\tht,\tht+\e)$ and $\sigg'\in\{-,+\}$.
    Similarly, for every $\tht\in\R$ for which $I^{L,\h}_{(\tht,-)}(s)$ is finite, there exists $\e>0$ such that $I^{L,\h}_{(\tht',\sig')}(s)=I^{L,\h}_{(\tht'',\sig'')}(s)$, for all $\tht',\tht''\in(\tht-\e,\tht)$ and $\sigg',\sigg''\in\{-,+\}$.
     \item\label{prop:general-I.i} For every $\tht\in\R$ for which $I^{R,\h}_{(\tht,-)}(s)$ is finite, there exists $\e>0$ such that $I^{R,\h}_{(\tht',\sig')}(s)=I^{R,\h}_{(\tht,-)}(s)$, for all $\tht'\in(\tht-\e,\tht)$ and $\sigg'\in\{-,+\}$.
    Similarly, for every $\tht\in\R$ for which $I^{R,\h}_{(\tht,+)}(s)$ is finite, there exists $\e>0$ such that $I^{R,\h}_{(\tht',\sig')}(s)=I^{R,\h}_{(\tht'',\sig'')}(s)$, for all $\tht',\tht''\in(\tht,\tht+\e)$ and $\sigg',\sigg''\in\{-,+\}$.
     \end{enumerate}
\end{prop}

\begin{proof}
If $\vartheta^{S,\h}_{(z,r)}=(\tht,\sigg)$, then $\geo\from{(z,r)}^{S,\h}$ coalesces with $W^{\tht\sig}$-geodesics. By Theorem \ref{thm:h-geo}\eqref{thm:h-geo.d}, the same coalescence holds for $\geo\from{(x,s)}^{S,\h}$, where $(x,s)=\geo\from{(z,r)}^{S,\h}(s)$. Thus, 
$\vartheta^{S,\h}_{(x,s)}=(\tht,\sigg)$. This shows that $\mathcal V^{S,\h}_r\subset\mathcal V^{S,\h}_s$, proving part \eqref{prop:general-I.a}.
\smallskip

By Theorem \ref{thm:h-geo}\eqref{thm:h-geo.a}, $\geo\from{(x,s)}^{L,\h}\preceq\geo\from{(x,s)}^{R,\h}$. Then, denoting $\vartheta^{L,\h}_{(x,s)}=(\tht_1,\sigg_1)$ and $\vartheta^{R,\h}_{(x,s)}=(\tht_2,\sigg_2)$, we have $\geo\from{(x,s)}^{L,\tht_1\sig_1}\preceq\geo\from{(x,s)}^{L,\tht_2\sig_2}$ and the ordering \eqref{geo:mono}, together with \eqref{sign}, implies $\vartheta_1\preceq\vartheta_2$. This proves part \eqref{prop:general-I.b}. The monotonicity claim in  part \eqref{prop:general-I.c} follows similarly, using  Theorem \ref{thm:h-geo}\eqref{thm:h-geo.f}. Then the definitions in \eqref{def:I-general} are precisely those of the right-continuous and left-continuous generalized inverses of $x\mapsto\vartheta^{S,\h}_{(x,s)}$, $S\in\{L,R\}$, respectively.
Part \eqref{prop:general-I.c} is proved. 
Part \eqref{prop:general-I.d} is immediate from the definition \eqref{def:I-general}.\smallskip

Part \eqref{prop:general-I.e}. By Theorem \ref{thm:h-geo}\eqref{thm:h-geo.j}, once $z<x$ is close enough to $x$, $\geo\from{(z,s)}^{L,\h}$ coalesces with $\geo\from{(x,s)}^{L,\h}$. This implies that $\vartheta^{L,\h}_{(z,s)}=\vartheta^{L,\h}_{(x,s)}$. 
This proves the first claim. The second claim also follows from Theorem \ref{thm:h-geo}\eqref{thm:h-geo.j}, as we have that once $y,y'\in(x,\infty)$ are close enough to $x$, $\geo\from{(y,s)}^{R,\h}$ coalesces with $\geo\from{(y',s)}^{R,\h}$.
Part \eqref{prop:general-I.f} is proved similarly.\smallskip

Part \eqref{prop:general-I.g}.  The continuity statements follow from part \eqref{prop:general-I.e}, and the inverse-function relation follows from the definition \eqref{def:I-general}. This then gives the relationship between the jump points and the ranges.

Part \eqref{prop:general-I.h}. Let $x=I^{L,\h}_{(\tht,+)}(s)$. Then for any $y>x$, $\vartheta^{L,\h}_{(y,s)}\succneq(\tht,+)$. This and the constancy in part \eqref{prop:general-I.e} imply the existence of $\tht_1>\tht$ and $\sigg_1\in\{-,+\}$ with $\vartheta^{L,\h}_{(y,s)}=(\tht_1,\sigg_1)$ for all $y>x$ close to $x$. Then for any $\tht'\in(\tht,\tht_1)$ and $\sigg'\in\{-,+\}$, we must have $I^{L,\h}_{(\tht',\sig')}(s)=x$. 
This proves the first claim. 

For the second claim, let $x=I^{L,\h}_{(\tht,-)}(s)$. Then, by the left continuity of $\vartheta^{L,\h}_{(\abullet,s)}$ from part \eqref{prop:general-I.e}, $\vartheta^{L,\h}_{(x,s)}\preceq(\tht,-)$. If $\vartheta^{L,\h}_{(x,s)}\precneq(\tht,-)$, then the same argument as above gives that for $\tht'<\tht$ close to $\tht$, the maximum $I^{L,\h}_{(\tht',\sig')}(s)$ remains at $x$ and we are done.
If, on the other hand, $\vartheta^{L,\h}_{(x,s)}=(\tht,-)$, then the constancy in part \eqref{prop:general-I.e} says that $\vartheta^{L,\h}_{(y,s)}=(\tht,-)$ for all $y<x$ close enough to $x$. Now, either $\vartheta^{L,\h}_{(y,s)}=(\tht,-)$ for all $y<x$, in which case $I^{L,\h}_{(\tht',\sig')}(s)=-\infty$ for all $\tht'<\tht$ and $\sigg'\in\{-,+\}$, and we are again done. Or 
$z=\sup\{y<x:\vartheta^{L,\h}_{(y,s)}\precneq(\tht,-)\}\in(-\infty,x)$,
in which case we must have $\vartheta^{L,\h}_{(z,s)}=(\tht_2,\sigg_2)$ with $\tht_2<\tht$ and $\sigg_2\in\{-,+\}$,   
and then for any $\tht'\in(\tht_2,\tht)$, $I^{L,\h}_{(\tht',\sig')}(s)=z$.  Part \eqref{prop:general-I.h} is proved. Part \eqref{prop:general-I.i} is proved similarly. 
\end{proof}

Our last result describes how the representation \eqref{h-rep} evolves backward in time.

\begin{thm}\label{thm:general-I}
There exists an event of full probability on which the following hold, for any eternal solution $\h$ and any $s\in\R$.
     \begin{enumerate} [label={\rm(\alph*)}, ref={\rm\alph*}]   \itemsep=3pt 
    \item\label{thm:general-I.a} For any $\vartheta\in\R\times\{-,+\}$, $S\in\{L,R\}$,  $\sigg\in\{-,+\}$, and $r\in\R$, $I^{S,\h}_\vartheta(s)=\sigg\infty$ is equivalent to $I^{S,\h}_\vartheta(r)=\sigg\infty$.
    \item\label{thm:general-I.b} For each $S\in\{L,R\}$ and any $\vartheta_1\neq\vartheta_2$ in $\R\times\{-,+\}$, $I^{S,\h}_{\vartheta_1}(s)=I^{S,\h}_{\vartheta_2}(s)$ implies $I^{S,\h}_{\vartheta_1}(r)=I^{S,\h}_{\vartheta_2}(r)$ for all $r\le s$. We then say that  $I^{S,\h}_{\vartheta_1}$ and  $I^{S,\h}_{\vartheta_2}$ coalesced. 
    \item\label{thm:general-I.c} For each $S \in \{L,R\}$ and any distinct $\vartheta_1,\vartheta_2 \in \R\times\{-,+\}$,  if both $I^{S,\h}_{\vartheta_1}$ and $I^{S,\h}_{\vartheta_2}$ are finite, then they eventually coalesce. Consequently, for any $\vartheta_1\precneq\vartheta_2$ in $\jumps^{S,\h}_s$, there exists $r < s$ such that  $\vartheta_2 \notin \jumps^{L,\h}_r$ and $\vartheta_1 \notin \jumps^{R,\h}_r$. 
    \item\label{thm:general-I.d} The representation \eqref{h-rep} can be restricted to $(\tht,\sigg)\in \jumps^{L,\h}_s$:
    \begin{align}\label{h-repV}
    h(x,s)=\sup_{(\tht,\siggg)\in\mathcal V^{S,\h}_s}\bigl\{W^{\tht\sig}(x,s;x_0,s_0)+a^{\tht\sig}_{(x_0,s_0)}(\h)\bigr\}
    \end{align}
More precisely, let $x_1<x_2$ be consecutive jump points of $x\mapsto \vartheta^{L,\h}_{(x,s)}$ {\rm(}equivalently, consecutive values of $I^{L,\h}_{\abullet}(s)${\rm)}, allowing $x_1=-\infty$ or $x_2=\infty$ if there is no preceding or succeeding jump. Set $(\tht,\sigg)=\vartheta^{L,\h}_{(x_2,s)}$, with the convention that if $x_2=\infty$, then $(\tht,\sigg)$ is the value attained after the last jump at $x_1$. Then
\begin{align}\label{h-rep4}
\h(x,s)=W^{\tht\sig}(x,s;x_0,s_0)+a^{\tht\sig}_{(x_0,s_0)}(\h),
\quad \text{for all } x \in (x_1,x_2],
\end{align}
where the endpoint $x_2$ is excluded when infinite. An analogous statement holds with $L$ replaced by $R$, $(x_1,x_2]$ replaced by $[x_1,x_2)$, and $(\tht,\sigg)=\vartheta^{R,\h}_{(x_1,s)}$, with the convention that if $x_1=-\infty$, the interval is $(-\infty,x_2)$ and $(\tht,\sigg)$ is the value before the first jump at $x_2$.
     \end{enumerate}
\end{thm}

\begin{rmk}\label{rk:countable}
From \eqref{h-rep} and \eqref{h-repV}, we have
\[h(x,s)=\sup_{(\tht,\siggg)\in\mathcal V^{S,\h}}\bigl\{W^{\tht\sig}(x,s;x_0,s_0)+a^{\tht\sig}_{(x_0,s_0)}(\h)\bigr\},\]
where $\mathcal V^{S,\h}=\bigcup_{r\in\R}\mathcal V^{S,\h}_r$.
Proposition \ref{prop:general-I}\eqref{prop:general-I.a} implies that this union  coincides with the union taken over integer times $r \in \Z$, and, by Proposition \ref{prop:general-I}\eqref{prop:general-I.e}, is therefore countable. Thus, one can reduce the representation \eqref{h-rep}, restricting to this countable set of $(\tht,\sigg)$ jump points. This proves the countability claim in Theorem \ref{thm:main1}.
\end{rmk}

\begin{rmk}\label{rk:h-finite?2}
Following Remark \ref{rk:h-finite?}, it is natural to ask for which choices of coefficients $a^{\tht\sig}$, $(\tht,\sigg)\in\R\times\{-,+\}$, satisfying $\sup_{\tht\sig} a^{\tht\sig}=0$, the function $(x,s)\mapsto \sup_{\tht\sig}\bigl\{W^{\tht\sig}(x,s;x_0,s_0)+a^{\tht\sig}\bigr\}$ takes finite values. The growth condition \eqref{h-growth} suggests that the coefficients $a^{\tht\sig}$ must grow sub-quadratically in $\tht$. A precise statement of this fact can be found in \cite[Proposition 4.3]{Bha-Bus-Sor-26-}.
\end{rmk}

\begin{rmk}
As in the proof of \cite[Lemma 6.14]{Ras-Swe-26-b-}, for $\vartheta\in\R\times\{-,+\}$ and $s\in\R$, the restriction $I^{L,\h}_\vartheta|_{(-\infty,s]}$ is a competition interface in the sense of \cite{Rah-Vir-25}. It follows that $I^{L,\h}_\vartheta$ is continuous. Although we do not use this continuity here, a proof may be found in the recent preprint \cite[Theorem 1.12]{Bha-Bus-Sor-26-}. 
\end{rmk}

\begin{proof}[Proof of Theorem \ref{thm:general-I}] We prove all the claims for the case $S=L$, the case $S=R$ being similar. 

Part \eqref{thm:general-I.a}. We work with $\sigg=+$, the other case being similar. $I^{L,\h}_\vartheta(s)=\infty$ is equivalent to $\vartheta^{L,\h}_{(x,s)}\preceq\vartheta$ for all $x\in\R$, which by part \eqref{prop:general-I.a} implies $\vartheta^{L,\h}_{(x,r)}\preceq\vartheta$ for all $r\le s$ and $x$ in $\R$ and hence $I^{L,\h}_\vartheta(r)=\infty$ for all such $r$. We prove that the same is true for all $r>s$.

Let $(\tht',\sigg')=\vartheta^{L,\h}_{(0,s)}$. Then,
by Proposition \ref{prop:general-I}\eqref{prop:general-I.c}, $\vartheta^{L,\h}_{(x,s)}\succeq(\tht',\sigg')$ for all $x>0$ and, by the ordering \eqref{geo:mono}, $\geo\from{(x,s)}^{L,\h}\succeq\geo\from{(x,s)}^{L,\tht'\sig'}$. 
By \cite[(6.7)]{Bus-Sep-Sor-24}, as $x\to\infty$, $\geo\from{(x,s)}^{L,\tht'\sig'}(r)\to\infty$
%
and hence also $\geo\from{(x,s)}^{L,\h}(r)\to\infty$. 
Take any $y\in\R$ and take $x>0$ large enough so that $\geo\from{(x,s)}^{L,\h}(r)>(y,r)$. Then the continuity of the paths and the coalescence \eqref{geo:coal1} imply that 
$\geo\from{(y,r)}^{L,\h}\preceq\geo\from{(x,s)}^{L,\h}$. 

If it were the case that $\vartheta^{L,\h}_{(y,r)}\succneq\vartheta$, then writing $(\tht_1,\sigg_1)=\vartheta^{L,\h}_{(x,s)}$ and $(\tht_2,\sigg_2)=\vartheta^{L,\h}_{(y,r)}$, we have $(\tht_1,\sigg_1)\preceq\vartheta\precneq(\tht_2,\sigg_2)$. 
Then \eqref{sign} (applied to $\tht_1$) and the monotonicity \eqref{geo:mono}
would make $\geo\from{(y,r)}^{L,\h}$ eventually go strictly right of $\geo\from{(x,s)}^{L,\h}$ and, by the continuity of the paths, the two would have to intersect. Theorem \ref{thm:h-geo}\eqref{thm:h-geo.e} then says the two must coalesce, which contradicts $\vartheta^{L,\h}_{(x,s)}\ne\vartheta^{L,\h}_{(y,r)}$. Thus, $\vartheta^{L,\h}_{(y,r)}\preceq\vartheta$, which implies $I^{L,\h}_{\vartheta}(r)\ge y$. Since $y$ was arbitrary, we get $I^{L,\h}_{\vartheta}(r)=\infty$. Part \eqref{thm:general-I.a} is proved.\smallskip

Part \eqref{thm:general-I.b}. 
We can assume, without loss of generality, that $\vartheta_1\precneq\vartheta_2$. 
Suppose $I^{L,\h}_{\vartheta_1}(s)=I^{L,\h}_{\vartheta_2}(s)$. 
By the right-continuity of $I^{L,\h}_{\abullet}(s)$, this is equivalent to it not having jump points 
$\vartheta$ with $\vartheta_1\precneq\vartheta\preceq\vartheta_2$.
Consequently, $\mathcal V^{L,\h}_s$ does not contain such $\vartheta$ points.
By part \eqref{prop:general-I.a}, this implies that the same holds for $\mathcal V^{L,\h}_r$, for any $r\le s$, which implies that $I^{L,\h}_{\vartheta_1}(r)=I^{L,\h}_{\vartheta_2}(r)$ for all such $r$.
\smallskip

Part \eqref{thm:general-I.c}. Suppose, without loss of generality, that $\vartheta_1\precneq\vartheta_2$ and that $I_1(r)=I^{L,\h}_{\vartheta_1}(r)<I^{L,\h}_{\vartheta_2}(r)=I_2(r)$ for all $r$. By assumption, $I_1$ and $I_2$ are finite (and by part \eqref{thm:general-I.a}, this holds for all $r$). Then Proposition \ref{prop:general-I}\eqref{prop:general-I.h} implies that the set of jump points $\tht$ of $I^{L,\h}_{\abullet}(r)$ 
satisfying $\vartheta_1\precneq\vartheta\preceq\vartheta_2$ is finite and hence closed. By parts \eqref{prop:general-I.a} and \eqref{prop:general-I.g} of that proposition, this set is nondecreasing in $r$. And under the assumption that $I_1(r)<I_2(r)$, this set is not empty. 
Therefore, there exists a $\vartheta$ with $\vartheta_1\precneq\vartheta\preceq\vartheta_2$ and such that $\vartheta$ is a jump point of $I^{L,\h}_{\abullet}(r)$, for all $r$. Let $z_r=I^{L,\h}_{\vartheta}(r)$.
Then $I_1(r)<z_r\le I_2(r)$
and $\vartheta^{L,\h}_{(z_r,r)}=\vartheta$.
Writing $\vartheta=(\tht,\sigg)$, we have that, for any $r\in\R$, $\geo\from{(z_r,r)}^{L,\h}$ coalesces with $W^{\tht\sig}$-geodesics. In particular, all these geodesics coalesce. However, by Theorem \ref{thm:h-geo}\eqref{thm:h-geo.d}, for any $t>r$, $\geo\from{(z_r,r)}^{L,\h}\bigl|_{[t,\infty)}$ is again an $\h$-geodesic that coalesces with $W^{\tht\sig}$-geodesics and hence $\vartheta^{L,\h}_{\geo\from{(z_r,r)}^{L,\h}(t)}=\vartheta$. This implies that $I_1(t)\le\geo\from{(z_r,r)}^{L,\h}(t)\le I_2(t)$. 
In other words, the coalescing geodesics  $\geo\from{(z_r,r)}^{L,\h}$ are all between $I_1$ and $I_2$. Then, there exists a subsequence $r_n\to-\infty$ such that $\geo\from{(z_{r_n},r_n)}^{L,\h}$ has a limit $\gamma$. This limit is now a bi-infinite geodesic, which is not allowed by \cite[Proposition 34]{Bha-24}. This contradiction proves the coalescence claim.  

If $\vartheta_1\precneq\vartheta_2$ are both in $\jumps^{L,\h}_s$, then once $I^{L,\h}_{\vartheta_1}$ and $I^{L,\h}_{\vartheta_2}$ coalesce, say at time $r<s$, we have $(\vartheta_1,\vartheta_2]\cap\jumps^{L,\h}_r=\varnothing$. In particular, $\vartheta_2\notin\jumps^{L,\h}_r$.
Part \eqref{thm:general-I.c} is proved.\smallskip

Part \eqref{thm:general-I.d}. Since $x_1$ and $x_2$ are consecutive jumps of $\vartheta^{L,\h}_{(\abullet,s)}$ and since the function is left-continuous, we must have 
$\vartheta^{L,\h}_{(x,s)}=(\tht,\sigg)$, for all $x\in(x_1,x_2]$. 
(This also works if $x_2=\infty$, in which case the interval is $(x_1,\infty)$.)
This means that $\geo\from{(x,s)}^{L,\h}$ coalesces with $W^{\tht\sigg}$-geodesics. Then, by the definition \eqref{def:a} (using $(y,t)=(x,s)$),
 $a^{\tht\sig}_{(x_0,s_0)}(\h)=\h(x,s)-W^{\tht\sig}(x,s;x_0,s_0)$.
Rearranging gives \eqref{h-rep4}.
\end{proof}

We conclude with a canonical version of \eqref{h-rep} whose coefficients are maximal and uniquely determined; see Remark \ref{rk:unique} for further discussion. The uniqueness statement in Theorem \ref{thm:h-rep-maximal}\eqref{thm:h-rep-maximal.d} also appears in \cite[Theorem 1.4]{Bha-Bus-Sor-26-}, as part of a stronger result: there, the correspondence between eternal solutions and their coefficients is shown to be a homeomorphism for suitable topologies, and the set of Busemann functions that actually contribute to the representation is characterized in terms of the coefficient function. That analysis relies on finer properties of the directed landscape, in particular competition interfaces and quantitative bounds on the Busemann process. The argument given here is tailored to uniqueness alone and is correspondingly shorter: the representation and its uniqueness follow from max-plus convex analysis together with the coalescence and monotonicity properties of Busemann geodesics already used above.

Let $\dirsp^\w$ be the quotient of $\R\times\{-,+\}$ obtained by identifying $(\tht,-)$ and $(\tht,+)$ whenever
$\tht\notin\baddir$, equipped with the order topology induced by the lexicographic order. For $\vartheta=(\tht,\sigg)\in\dirsp^\w$, abbreviate $W^{\vartheta}=W^{\tht\sig}$ and $\acoef^{\vartheta}_{(x_0,s_0)}(\h)=\acoef^{\tht\sig}_{(x_0,s_0)}(\h)$. Given a function $a:\dirsp^\w\to\R\cup\{-\infty\}$, its \emph{upper semicontinuous regularization} is given by 
\begin{align}\label{def:usc}
a^{\mathrm{usc}}(\vartheta)=\inf_{U\ni\vartheta}\sup_{\vartheta'\in U}a(\vartheta'),
\end{align}
where the infimum is over the open neighborhoods $U$ of $\vartheta$ in $\dirsp^\w$.

For an eternal solution $\h$ with $\h(x_0,s_0)=0$, define
\begin{align}\label{def:acoef-DL}
\acoef^{\vartheta}_{(x_0,s_0)}(\h)
=\inf_{(x,s)\in\R^2}
\bigl\{\h(x,s)-W^{\vartheta}(x,s;x_0,s_0)\bigr\}.
\end{align}
The quantity inside the infimum is real-valued and equals zero at
$(x_0,s_0)$, so the infimum belongs to $\R\cup\{-\infty\}$ and
is at most zero. Equivalently,
\[\acoef^{\vartheta}_{(x_0,s_0)}(\h)
=\sup\bigl\{
a\in\R\cup\{-\infty\}:
W^{\vartheta}(\abullet;x_0,s_0)+a\le\h\bigr\}\in[-\infty,0].\]

\begin{thm}\label{thm:h-rep-maximal}
The following holds on a full probability event. Let $\h$ be an
eternal solution that vanishes at $(x_0,s_0)\in\R^2$.
\begin{enumerate}[label={\rm(\alph*)}, ref={\rm\alph*}]
\item\label{thm:h-rep-maximal.a}
For all $(x,s)\in\R^2$,
\begin{align}\label{h-rep-maximal}
\h(x,s)=\sup_{\vartheta\in\dirsp^\w}
\bigl\{W^{\vartheta}(x,s;x_0,s_0)+\acoef^{\vartheta}_{(x_0,s_0)}(\h)\bigr\}.
\end{align}

\item\label{thm:h-rep-maximal.b}
The coefficients $\acoef^{\vartheta}_{(x_0,s_0)}(\h)$ are
maximal: if $a:\dirsp^\w\to\R\cup\{-\infty\}$ satisfies
\begin{align}\label{other-rep}
\h=\sup_{\vartheta\in\dirsp^\w}
\bigl\{W^{\vartheta}(\abullet;x_0,s_0)+a(\vartheta)\bigr\},
\end{align}
then $a(\vartheta)\le\acoef^{\vartheta}_{(x_0,s_0)}(\h)$
for every $\vartheta\in\dirsp^\w$.
Consequently, \eqref{h-rep-maximal} is the unique representation
over the full family of Busemann functions in which every
coefficient is maximal.

\item\label{thm:h-rep-maximal.c}
The map $\dirsp^\w\ni\vartheta
\longmapsto\acoef^\vartheta_{(x_0,s_0)}(\h)
\in\R\cup\{-\infty\}$
is upper semicontinuous.

\item\label{thm:h-rep-maximal.d} Let $a:\dirsp^\w\to\R\cup\{-\infty\}$ be any coefficients satisfying
\begin{align}\label{any-rep}
\h=\sup_{\vartheta\in\dirsp^\w}\bigl\{W^\vartheta(\abullet;x_0,s_0)+a(\vartheta)\bigr\}.
\end{align}
Then $a^{\mathrm{usc}}(\vartheta)=\acoef^{\vartheta}_{(x_0,s_0)}(\h)$ for all $\vartheta\in\dirsp^\w$. In particular, $a^{\mathrm{usc}}$  satisfies \eqref{any-rep}, and \eqref{h-rep-maximal} is the unique representation of $\h$ whose coefficients are upper semicontinuous on $\dirsp^\w$.
\end{enumerate}
\end{thm}

\begin{proof}
From \eqref{other-rep}, we get $a(\vartheta)\le\h(x,s)-W^{\vartheta}(x,s;x_0,s_0)$ for all $(x,s)\in\R^2$.
Taking the infimum over $(x,s)$ proves the maximality in
\eqref{thm:h-rep-maximal.b}. Applying this to \eqref{h-rep} gives
$a^{\vartheta}_{(x_0,s_0)}(\h)\le\acoef^{\vartheta}_{(x_0,s_0)}(\h)$,
and therefore
\[\h\le\sup_{\vartheta\in\dirsp^\w}
\bigl\{W^{\vartheta}(\abullet;x_0,s_0)
+\acoef^{\vartheta}_{(x_0,s_0)}(\h)\bigr\}.\]
(Note that when $\tht\notin\baddir$, \eqref{def:a} gives $a^{\tht-}_{(x_0,s_0)}(\h)=a^{\tht+}_{(x_0,s_0)}(\h)$. Hence $a^{\vartheta}_{(x_0,s_0)}(\h)$ is well defined for $\vartheta\in\dirsp^\w$.)
Conversely, the definition \eqref{def:acoef-DL} gives
\[W^{\vartheta}(\abullet;x_0,s_0)+\acoef^{\vartheta}_{(x_0,s_0)}(\h)\le\h\]
for every $\vartheta\in\dirsp^\w$. Taking the supremum proves the reverse inequality and hence \eqref{h-rep-maximal}. The maximality just proved also gives the uniqueness claim.

By \cite[Theorem 5.1(iv)]{Bus-Sep-Sor-24}, for each
fixed $(x,s)\in\R^2$, the map $\vartheta\mapsto\h(x,s)-W^{\vartheta}(x,s;x_0,s_0)$ is continuous on $\dirsp^\w$. Since $\acoef^{\vartheta}_{(x_0,s_0)}(\h)$ is the infimum of these continuous functions over $(x,s)\in\R^2$, it is upper
semicontinuous. This proves \eqref{thm:h-rep-maximal.c}.
It remains to prove part \eqref{thm:h-rep-maximal.d}.

$\h(x_0,s_0)=0$ and \eqref{any-rep} give $\sup_\vartheta a(\vartheta)=0$. As $\h$ and $(x_0,s_0)$ are fixed, we will write $W^{\vartheta}(x,s)=W^{\vartheta}(x,s;x_0,s_0)$ and $\acoef^{\vartheta}=\acoef^{\vartheta}_{(x_0,s_0)}(\h)$ for the rest of the proof.

By part \eqref{thm:h-rep-maximal.b}, $a\le\acoef$, and by part \eqref{thm:h-rep-maximal.c}, $\acoef$ is upper semicontinuous on $\dirsp^\w$. Since $a^{\mathrm{usc}}$ is the smallest upper semicontinuous majorant of $a$, we get $a\le a^{\mathrm{usc}}\le\acoef$. This and \eqref{h-rep-maximal} give 
\[\h=\sup_{\vartheta}\{W^{\vartheta}+a(\vartheta)\}\le\sup_{\vartheta}\{W^{\vartheta}+a^{\mathrm{usc}}(\vartheta)\}\le\sup_{\vartheta}\{W^{\vartheta}+\acoef^{\vartheta}\}=\h.\]
Thus $a^{\mathrm{usc}}$ satisfies \eqref{any-rep}. 

Fix $\vartheta_0=(\tht_0,\sigg_0)\in\dirsp^\w$. We prove that 
\begin{align}\label{1436}
\acoef^{\vartheta_0}\le a^{\mathrm{usc}}(\vartheta_0).
\end{align}
There is nothing to prove if $\acoef^{\vartheta_0}=-\infty$, so assume $c=\acoef^{\vartheta_0}>-\infty$. For $\vartheta\in\dirsp^\w$ and $t\ge s_0$, set
\[\Delta_\vartheta(t)=W^{\vartheta_0}(\geo\from{(x_0,s_0)}^{L,\vartheta_0}(t))-W^{\vartheta}(\geo\from{(x_0,s_0)}^{L,\vartheta_0}(t)).\]

\begin{lem}\label{lm:Delta}
$\Delta_\vartheta$ satisfies the following properties.
\begin{enumerate} [label={\rm(\roman*)}, ref={\rm\roman*}]   \itemsep=3pt
\item\label{Delta.i} $\Delta_\vartheta(t)\ge\Delta_\vartheta(t')\ge0$ for all $\vartheta\in\dirsp^\w$ and $t>t'\ge s_0$.
\item\label{Delta.ii} For $\vartheta\ne\vartheta_0$, $\Delta_\vartheta(t)\to\infty$ as $t\to\infty$.
\item\label{Delta.iii} If $\vartheta_1\precneq\vartheta_0\precneq\vartheta_2$ in $\dirsp^\w$, then for all large $t$,
\begin{align}\label{aux:Delta-unif}
\inf_{\vartheta\succeq\vartheta_2}\Delta_\vartheta(t)\ge\Delta_{\vartheta_2}(t)
\quad\text{and}\quad
\inf_{\vartheta\preceq\vartheta_1}\Delta_\vartheta(t)\ge\Delta_{\vartheta_1}(t).
\end{align}
\end{enumerate}
\end{lem}

\begin{proof}
Use \eqref{W-eternal}, the cocycle property \eqref{cocycle}, and the fact that $\geo\from{(x_0,s_0)}^{L,\vartheta_0}$ is a $W^{\vartheta_0}$-geodesic to write
\[W^{\vartheta_0}(\geo\from{(x_0,s_0)}^{L,\vartheta_0}(t');\geo\from{(x_0,s_0)}^{L,\vartheta_0}(t))=\LL(\geo\from{(x_0,s_0)}^{L,\vartheta_0}(t');\geo\from{(x_0,s_0)}^{L,\vartheta_0}(t))
\le W^{\vartheta}(\geo\from{(x_0,s_0)}^{L,\vartheta_0}(t');\geo\from{(x_0,s_0)}^{L,\vartheta_0}(t)).\]
Rearrange and use the cocycle property \eqref{cocycle} again to get the first inequality in part \eqref{Delta.i}. Observing that $\Delta_\vartheta(s_0)=0$ gives the second inequality. In particular, $\lim_{t\to\infty}\Delta_\vartheta(t)$ exists in $[0,\infty]$.

Now fix $(x,s)\in\R^2$. By \eqref{geo:coal1} and \eqref{geo:coal2}, $\geo\from{(x_0,s_0)}^{L,\vartheta_0}(t)$ lies on $\geo\from{(x,s)}^{L,\vartheta_0}$ for all large $t$. For such $t$, since $\geo\from{(x,s)}^{L,\vartheta_0}$ is a $W^{\vartheta_0}$-geodesic,
$W^{\vartheta_0}(x,s;\geo\from{(x_0,s_0)}^{L,\vartheta_0}(t))=\LL(x,s;\geo\from{(x_0,s_0)}^{L,\vartheta_0}(t))\le W^{\vartheta}(x,s;\geo\from{(x_0,s_0)}^{L,\vartheta_0}(t))$,
and the cocycle property \eqref{cocycle} gives
$\Delta_\vartheta(t)
\ge W^{\vartheta_0}(x,s)-W^{\vartheta}(x,s)$.
Letting $t\to\infty$ and then taking the supremum over $(x,s)\in\R^2$ and using  Lemma \ref{aux:349} gives part \eqref{Delta.ii}.

Lastly, we prove the first inequality in \eqref{aux:Delta-unif}, the second being symmetric. For large $t$, we have $\geo\from{(x_0,s_0)}^{L,\vartheta_0}(t)\le\geo\from{(x_0,s_0)}^{L,\vartheta_2}(t)$. 
To see this, write $\vartheta_2=(\theta_2,\sigg_2)$. If $\tht_0<\tht_2$, then the claim follows from the $\tht_0$- and $\tht_2$-directedness of $\geo\from{(x_0,s_0)}^{L,\vartheta_0}$ and $\geo\from{(x_0,s_0)}^{L,\vartheta_2}$ in \cite[Theorem 5.9(ii)(d)]{Bus-Sep-Sor-24}, while if $\tht_0=\tht_2$ then $\tht_0\in\baddir$, $\sigg_0=-$, $\sigg_2=+$, and the claim is \eqref{sign}. Fix such a $t$ and $\vartheta\succeq\vartheta_2$. By the cocycle property \eqref{cocycle}, the inequality $W^{\vartheta}(x_0,s_0;\geo\from{(x_0,s_0)}^{L,\vartheta_2}(t))\ge\LL(x_0,s_0;\geo\from{(x_0,s_0)}^{L,\vartheta_2}(t))$ (which follows from \eqref{W-eternal}), the monotonicity \eqref{W:mono}, and the fact that $\geo\from{(x_0,s_0)}^{L,\vartheta_2}$ is a $W^{\vartheta_2}$-geodesic,
\begin{align*}
&W^{\vartheta}(x_0,s_0;\geo\from{(x_0,s_0)}^{L,\vartheta_0}(t))
=W^{\vartheta}(x_0,s_0;\geo\from{(x_0,s_0)}^{L,\vartheta_2}(t))+W^{\vartheta}(\geo\from{(x_0,s_0)}^{L,\vartheta_2}(t);\geo\from{(x_0,s_0)}^{L,\vartheta_0}(t))\\
&\qquad\ge\LL(x_0,s_0;\geo\from{(x_0,s_0)}^{L,\vartheta_2}(t))+W^{\vartheta_2}(\geo\from{(x_0,s_0)}^{L,\vartheta_2}(t);\geo\from{(x_0,s_0)}^{L,\vartheta_0}(t))\\
&\qquad=W^{\vartheta_2}(x_0,s_0;\geo\from{(x_0,s_0)}^{L,\vartheta_2}(t))+W^{\vartheta_2}(\geo\from{(x_0,s_0)}^{L,\vartheta_2}(t);\geo\from{(x_0,s_0)}^{L,\vartheta_0}(t))
=W^{\vartheta_2}(x_0,s_0;\geo\from{(x_0,s_0)}^{L,\vartheta_0}(t)).
\end{align*}
Subtracting $W^{\vartheta_0}(x_0,s_0;\geo\from{(x_0,s_0)}^{L,\vartheta_0}(t))$ from both sides and using \eqref{cocycle} gives $\Delta_\vartheta(t)\ge\Delta_{\vartheta_2}(t)$.
\end{proof}

We can now conclude the proof of the theorem. Since $\dirsp^\w$ has neither a least nor a greatest element, every open neighborhood of $\vartheta_0$ contains an open interval $U=\{\vartheta:\vartheta_1\precneq\vartheta\precneq\vartheta_2\}$ with $\vartheta_1\precneq\vartheta_0\precneq\vartheta_2$. By Lemma \ref{lm:Delta}, there exists $t\ge s_0$ such that $\inf_{\vartheta\notin U}\Delta_\vartheta(t)>-c$. For this $t$, the definition \eqref{def:acoef-DL}, the representation \eqref{any-rep}, the nonnegativity of $\Delta_\vartheta$, and $a\le0$ give
\begin{align*}
c=\acoef^{\vartheta_0}\le\h(\geo\from{(x_0,s_0)}^{L,\vartheta_0}(t))-W^{\vartheta_0}(\geo\from{(x_0,s_0)}^{L,\vartheta_0}(t))
=\sup_{\vartheta\in\dirsp^\w}\bigl\{a(\vartheta)-\Delta_\vartheta(t)\bigr\}
\le\max\Bigl\{\sup_{\vartheta\in U}a(\vartheta)\,,\;\sup_{\vartheta\notin U}\{-\Delta_\vartheta(t)\}\Bigr\}.
\end{align*}
Since the second entry of the maximum is $<c$, we must have $\sup_{U}a\ge c$. As this holds for every open interval $U\ni\vartheta_0$, and hence for every open neighborhood of $\vartheta_0$, \eqref{def:usc} gives $a^{\mathrm{usc}}(\vartheta_0)\ge c=\acoef^{\vartheta_0}$. We have thus shown that $a^{\mathrm{usc}}=\acoef$.
\end{proof}


\begin{rmk}\label{rk:unique}
The coefficients in \eqref{h-rep} satisfy
$a^{\tht\sig}_{(x_0,s_0)}(\h)=\acoef^{\tht\sig}_{(x_0,s_0)}(\h)$ whenever $a^{\tht\sig}_{(x_0,s_0)}(\h)>-\infty$. Indeed, in this case there exists $(y,t)\in\R^2$ such that
\[a^{\tht\sig}_{(x_0,s_0)}(\h)=\h(y,t)-W^{\tht\sig}(y,t;x_0,s_0)\ge\acoef^{\tht\sig}_{(x_0,s_0)}(\h),\]
with the reverse inequality coming from the maximality in Theorem \ref{thm:h-rep-maximal}\eqref{thm:h-rep-maximal.b}.
Nevertheless, strict inequality can occur. Indeed, consider the eternal solution $\h=\sup_{\tht\in[-1,1]}W^{\tht-}(\abullet;x_0,s_0)$.
For every $\tht\in[-1,1]$, $W^{\tht-}(\abullet;x_0,s_0)\le\h$, while both functions vanish
at $(x_0,s_0)$. Hence, $\acoef^{\tht-}_{(x_0,s_0)}(\h)=0$
for every $\tht\in[-1,1]$.
On the other hand, Remark \ref{rk:countable} implies that
$a^{\tht\sig}_{(x_0,s_0)}(\h)>-\infty$ for only countably many
$(\tht,\sigg)$. Consequently, for uncountably many
$\tht\in[-1,1]$, $a^{\tht-}_{(x_0,s_0)}(\h)=-\infty<
0=\acoef^{\tht-}_{(x_0,s_0)}(\h)$.
\end{rmk}

\appendix

\section{Continuity of the KPZ fixed point}

First, we establish a max-plus analogue of the dominated convergence theorem.

\begin{thm}\label{thm:domcv}
Let $d$ be a positive integer and let $S\subset\R^d$ be a nonempty subset.
Let $f_n,f:S\to\R$ be functions. Assume that
$f_n\to f$ locally uniformly on $S$, i.e.\ for every compact $K\subset\R^d$, $f_n\to f$ uniformly on $S\cap K$. Assume further that there exists a function $g:S\to\R$ such that
$f_n(x)\le g(x)$ for all $x\in S$ and all $n$ and 
$g(x)\to -\infty$ as $|x|\to\infty$ within $S$.
Then
\[
\lim_{n\to\infty}\sup_{x\in S} f_n(x)=\sup_{x\in S} f(x).
\]
\end{thm}

\begin{proof}
Set
\[M_n=\sup_{x\in S} f_n(x) \quad\text{and}\quad
    M=\sup_{x\in S} f(x).\]
Since $f_n(x)\leq g(x)$ for all $n$ and $f_n(x)\to f(x)$ pointwise, we have
$f(x)\leq g(x)$ for all $x\in S$.
Let $a<M$. Then there exists $x_0\in S$
such that $f(x_0)>a$. Since $f_n(x_0)\to f(x_0)$, we have
$f_n(x_0)>a$ for all sufficiently large $n$. Hence
$\varliminf_{n\to\infty} M_n\geq a$.
Take $a\nearrow M$ to get 
$\varliminf_{n\to\infty} M_n\geq M$.

The upper bound is trivial when $M=\infty$. Assume therefore that $M<\infty$. Fix $a<M$. Since
$g(x)\to-\infty$ as $|x|\to\infty$ within $S$, choose $R>0$ such that
$g(x)\leq a$ for all $x\in S$ with $|x|>R$. 
Then also $f(x)\leq a$ for all such $x$, and therefore
\[M=\sup_{x\in S:\, |x|\leq R} f(x).\]
Moreover, for every $n$,
\[
    M_n
    \leq
    \max\Bigl\{
        \sup_{x\in S:\, |x|\leq R} f_n(x),
        \sup_{x\in S:\, |x|>R} g(x)
    \Bigr\}
    \leq
    \max\Bigl\{
        \sup_{x\in S:\, |x|\leq R} f_n(x),
        a
    \Bigr\}.
\]
Since $f_n\to f$ uniformly on $S\cap\{|x|\leq R\}$,
\[
    \sup_{x\in S:\, |x|\leq R} f_n(x)
    \longrightarrow
    \sup_{x\in S:\, |x|\leq R} f(x)
    =
    M.
\]
Hence
\[
    \varlimsup_{n\to\infty} M_n
    \leq
    \max\{M,a\}
    =
    M.
\]
Combining this with the lower bound gives $M_n\to M$.
\end{proof}

Now we apply the above theorem to prove the continuity of the KPZ fixed point.

\begin{thm}\label{thm:h-continuous}
The following holds on a full probability event. For any $t>s$ and any function $\h:\R\to[-\infty,\infty)$ satisfying 
\begin{align}\label{h-growth2}
    \lim_{|y|\to\infty} |y|^{-1}\bigl(\h(y)-y^2/(t-s)\bigr) = -\infty,
\end{align}
$(x,r)\mapsto\sup_{y\in\R}\{\mathcal L(x,r;y,t)+\h(y)\}$ is continuous on $\R\times[s,t)$, with values in $[-\infty,\infty]$.
\end{thm}

\begin{proof}
Let $S=\{x:\h(x)>-\infty\}$. If $S=\varnothing$, then the claim holds trivially. Assume $S$ is not empty.
 Take $(x_n,r_n)\to(x,r)$ in $\R\times[s,t)$. Let $f_n(y)=\mathcal L(x_n,r_n;y,t)+\h(y)$ and 
 $f(y)=\mathcal L(x,r;y,t)+\h(y)$. The continuity of $\mathcal L$ implies that $f_n\to f$ locally uniformly on $S$, i.e.\ the convergence is uniform on the intersection of $S$ with any compact subset of $\R$. 

There exist a compact interval $K\subset\R$, a number $r_1\in[s,t)$, and $n_0$ such that $x_n,x\in K$ and $r_n,r\in[s,r_1]$, for all $n\ge n_0$. 
Define
\[g(y)=\h(y)+\sup_{(z,q)\in K\times[s,r_1]}\mathcal L(z,q;y,t).\]
Then for every $n\ge n_0$ and every $y\in S$, $f_n(y)\le g(y)$.
Moreover, \eqref{L-bound}, together with \eqref{h-growth2},
implies that
$g(y)\to -\infty$ as $|y|\to\infty$ in $S$.
Therefore, Theorem \ref{thm:domcv} applies, and yields
\[\sup_{x\in\R} f_n(x)=\sup_{x\in S}f_n(x)\to \sup_{x\in S} f(x)=\sup_{x\in\R}f(x).\qedhere\]
\end{proof}

\section{An auxiliary lemma}

\begin{lem}\label{lm9.1}
The following holds on a full probability event.
For any $\tht\in\R$, any compact set $K\subset\R^2$, and any $p=(x,s)\in\R^2$ and $T>s\vee\sup\{t:(y,t)\in K\}$, there exists a finite time $T'>T$ such that for any $(x,s)\in K$, $r\ge T'$, and any $z\in\R$ with $\geo_p^{L,\tht-}(r)\le(z,r)\le\geo_p^{R,\tht+}(r)$, we have,
\begin{align}\label{lm9.1:eqn}
\geo^L_{(x,s):(z,r)}\big|_{[s,T]} \in\Bigl\{\geo^{L,\tht-}_{(x,s)}\big|_{[s,T]}\,,\,\geo^{L,\tht+}_{(x,s)}\big|_{[s,T]}\Bigr\}
            \text{ and }\ 
			\geo^R_{(x,s):(z,r)}\big|_{[s,T]}
            \in
            \Bigl\{\geo^{R,\tht-}_{(x,s)}\big|_{[s,T]}\,,\,\geo^{R,\tht+}_{(x,s)}\big|_{[s,T]}\Bigr\}.
\end{align}
\end{lem}

\begin{proof}
Suppose there exists a strictly increasing sequence $r_n>T$ and sequences $(x_n,s_n)\in K$ and $z_n$ such that $\geo_p^{L,\tht-}(r_n)\le(z_n,r_n)\le\geo_p^{R,\tht+}(r_n)$ for all $n$ and
for which \eqref{lm9.1:eqn} fails. 
By the $\tht$-directedness of $\geo_p^{L,\tht\pm}$ in \cite[Theorem 5.9(ii)(d)]{Bus-Sep-Sor-24}, $z_n/r_n\to\tht$.  Now, we have a contradiction with 
 \cite[Lemma 9.1]{Bha-Bus-Sor-25-}.\smallskip

We briefly sketch the contradiction argument, for completeness. Suppose the first claim in \eqref{lm9.1:eqn} fails. Since all geodesics involved are leftmost geodesics, this is equivalent to
\begin{align}\label{fail}
\geo_{(x_n,s_n):(z_n,r_n)}(T)\notin\{\geo_{(x_n,s_n)}^{L,\tht-}(T),\geo_{(x_n,s_n)}^{L,\tht+}(T)\}.
\end{align}
By compactness, we may assume $(x_n,s_n)\to(x,s)$ and $\geo^L_{(x_n,s_n):(z_n,r_n)}$ converges to a semi-infinite geodesic $\tau$ from $(x,s)$. Fix $s'>s$ and take $n$ large enough that $s_n<s'$. By planarity, after passing to a subsequence, $\geo^L_{(x_n,s_n):(z_n,r_n)}|_{[s',\infty)}$ is monotone, and hence converges uniformly by Dini's theorem. Theorem 1.18 in \cite{Bat-Gan-Ham-22} then implies convergence in the overlap topology, so for all large $n$,
$\geo^L_{(x_n,s_n):(z_n,r_n)}|_{[s',T]}=\tau|_{[s',T]}$.
Thus \eqref{fail} yields
\begin{align}\label{1498}
\tau(T)\notin\{\geo^{L,\tht-}_{\tau(s')}(T),\geo^{R,\tht+}_{\tau(s')}(T)\}.
\end{align}
On the other hand, since $\geo_p^{L,\tht-}\preceq(z_n,r_n)\preceq\geo_p^{R,\tht+}$ for all $n$, $\tau$ is also eventually between the two geodesics and is thus $\tht$-directed. Therefore $\tau$ must be a $W^{\tht\sig}$-geodesic for $\sigg\in\{-,+\}$. But coalescence \eqref{geo:coal1} implies that there is a unique $W^{\tht\sig}$-geodesic out of $\tau(s')$, contradicting \eqref{1498}.
\end{proof}

\bibliographystyle{aop-no-url}
\bibliography{firasbib2010}

\end{document}